\documentclass[reqno,11pt]{amsart}

\usepackage{amssymb,amsmath,amsthm,amsfonts,xcolor} 

\usepackage[left=3cm,right=3cm,top=3cm,bottom=3cm]{geometry} 
\usepackage{comment,graphicx,setspace}

\usepackage[colorlinks=true, pdfstartview=FitV, linkcolor=blue, citecolor=blue, urlcolor=blue]{hyperref}

\allowdisplaybreaks
\numberwithin{equation}{section}

\renewcommand{\div}{\mathrm{div}}

\newcommand{\N}{\mathbb{N}}
\newcommand{\R}{\mathbb{R}}
\newcommand{\C}{\mathbb{C}}
\renewcommand{\S}{\mathbb{S}}

\newcommand{\cA}{\mathcal{A}}
\newcommand{\cB}{\mathcal{B}}
\newcommand{\cC}{\mathcal{C}}
\newcommand{\cD}{\mathcal{D}}
\newcommand{\cF}{\mathcal{F}}
\newcommand{\cG}{\mathcal{G}}
\newcommand{\cH}{\mathcal{H}}
\newcommand{\cL}{\mathcal{L}}
\newcommand{\cN}{\mathcal{N}}

\newcommand{\cR}{\mathcal{R}}
\newcommand{\cS}{\mathcal{S}}

\newcommand{\bI}{\mathbf{I}}
\newcommand{\bO}{\mathbf{O}}
\newcommand{\bQ}{\mathbf{Q}}

\newcommand{\p}{\partial}

\newcommand{\ri}{{\rm i}}

\newcommand{\cl}[1]{\overline{#1}} 
\newcommand{\dist}{{\rm dist}}

\newcommand{\Lip}{{\rm Lip}}
\newcommand{\holder}{C}
\newcommand{\networks}{\mathcal{X} }

\newcommand{\pnt}{q}

\renewcommand{\hat}{\widehat}

\def\restriction#1|#2{{#1}_{\lower2pt\hbox{$\scriptstyle\vert{#2}$}}}

\newtheorem{theorem}{Theorem}[section]
\newtheorem{proposition}[theorem]{Proposition}
\newtheorem{lemma}[theorem]{Lemma}
\newtheorem{corollary}[theorem]{Corollary}

\theoremstyle{definition}
\newtheorem{definition}[theorem]{Definition}
\newtheorem{assumption}[theorem]{Assumption}

\theoremstyle{remark}
\newtheorem{remark}[theorem]{Remark}

\newtheorem{example}[theorem]{Example}

\title{Some aspects of anisotropic curvature flow of planar partitions}

\author[G. Bellettini] {Giovanni Bellettini} 
\address[G. Bellettini]{University of Siena \\ via Roma 56,  53100 Siena, Italy \\ \& \\ International Centre for Theoretical Physics (ICTP)\\ Strada Costiera 11, 34151 Trieste (Italy)}
\email{bellettini@diism.unisi.it}

\author[Sh. Kholmatov] {Shokhrukh Yu. Kholmatov} 
\address[Sh. Kholmatov]{Fakult\"at f\"ur Mathematik\\ 
Universit\"at Wien\\ Oskar-Morgenstern Platz 1\\1090 Wien 
(Austria)}
\email{shokhrukh.kholmatov@univie.ac.at}

\keywords{anisotropy, network, partition, triple junctions, crystalline curvature, curvature flow}

\subjclass[2020]{53E10, 35D30, 35D35}

\date{\today}

\begin{document}

\begin{abstract}
We consider the geometric evolution of a network in the plane, flowing by anisotropic curvature. We discuss local existence of a classical solution in the presence of several smooth anisotropies. Next, we discuss some aspects
of the polycrystalline case. 
\end{abstract}

\maketitle

\section{Introduction}

Many processes in material sciences such as phase transformation, crystal growth, domain growth, grain growth, ion beam and chemical etching, {\it etc.} can be modelled as a {\bf geometric interface motion} in which surface tension acts  as a principal driving force (see e.g. 
\cite{BBP:2001,BBBP:1997,Cahn:1991,CHT:1992,CK:1994,GGM:1998,Herring:1999,KL:2001,Mullins:1956,Taylor:1978,TCH:1992} and references therein). An interface (or surface boundary) in the plane is a curve bounding different 
regions (phases) and moving in a nonequilibrium state \cite{CDR:1989,DKSch:1993,Herring:1951,Schl:2000}. 

In some simplified cases 
the motion of this curve does not depend on the physical
situation in the various phases\footnote{The case
of two phases is usually called anisotropic curve shortening flow, and will not be addressed here;
we refer the reader to \cite{Andrews:1998,AB:2011}  and to \cite{Andrews:2002,GP:2022} in the crystalline case.},  and is described by geometric equations relating, for instance, the normal velocity of the interface to its curvature. 
The anisotropic curvature flow in two dimensions 
of a network $\Sigma$
is the formal gradient flow of the energy functional
$$
\ell_\phi(\Sigma):=\int_\Sigma  \phi^o (\nu_\Sigma)d\cH^1,
$$
where $\Sigma$ is a set of curves delimitating the various phases, and typically having triple
junctions, $\nu_\Sigma$ is a unit normal vector field to $\Sigma$ 
and the energy density $ \phi^o : \R^2\to[0,+\infty)$, 
sometimes called surface tension (or, generally,  anisotropy), 
is defined on $\mathbb S^1$ 
and its one-homogeneous extension on $ \R^2$ is
a norm. An interesting case is when $ \phi^o $ is crystalline, i.e., its unit ball $B_{\phi^o}$ is a (centrally symmetric) polygon. In such a case, one expects the phases to be mostly polygonal regions, which evolve under a sort 
of nonlocal curvature flow\footnote{The interest
is due mainly to the presence of facets and corners in $\p B_{\phi^o}$. However,
a mathematical obstruction is
represented by the possible appearence of nonpolygonal curves during the crystalline flow of a
network, arising from triple junctions.}.
More realistic is the case in which various anisotropies
are involved in the energy, i.e., $ \phi^o _{ij}$ 
is an anisotropy weighting the part of the network $\Sigma$
dividing phase $i$ from phase $j$. When all $ \phi^o _{ij}$ 
are crystalline, this is a model for polycrystalline materials \cite{ELM:2021,GN:2000}.

The aim of this paper is to 
discuss some aspects of the evolution of the network $\Sigma$ under
anisotropic curvature flow; for simplicity we do not include mobilities. We quickly review some known results when $ \phi^o $ is Euclidean, and discuss some aspects of the flow in the anisotropic and polycrystalline cases,
starting from the definition of what we mean by normal velocity. We will give some detail on the short time existence of a strong 
solution when $ \phi^o _{ij}$ are smooth and uniformly convex. The variational nature of the flow will be emphasized.

Nothing will be specified in this paper for weak
solutions to the flow: for this argument we refer
the reader to \cite{BH:2018.siam,BChKh:2020,LO:2016,Tonegawa:2019}.

G.B. is very grateful to Errico Presutti for having shared his deep knowledge on some aspects of mathematical physics and, above all, for his generosity.

{\bf Acknowledgements.} Sh. Kholmatov acknowledges support from the Austrian Science Fund (FWF) Stand-Alone project P 33716. G. Bellettini is a member of the GNAMPA (INdAM) of Italy.

\section{Notation}

We denote by $\cdot$ the Euclidean scalar product 
and by $|\cdot|$ the Euclidean norm in $ \R^2.$ Given $a,b\in \R^2,$ 
 $a\otimes b$ stands for the $2\times2$-matrix with entries 
$(a\otimes b)_{ij} = a_ib_j.$ The symbol $\cH^1$ stands for the $1$-dimensional Hausdorff measure in $ \R^2.$  We denote by $a^\perp$ the counterclockwise $90^o$-rotation of a nonzero vector $a\in \R^2,$ i.e., 
$$
a = (a_1,a_2)\quad\Longrightarrow \quad a^\perp = (-a_2,a_1).
$$
The (topological) boundary of a set $E\subset \R^2$ is denoted by $\p E.$

We identify both  tangent and cotangent spaces at a point 
of $ \R^2$ with (a copy of) $ \R^2.$

\subsection{Anisotropies}
We denote by $\phi: \R^2\to[0,+\infty)$ an anisotropy, 
i.e., a convex function such that 
$$
\phi(\lambda \xi) = |\lambda|\phi(\xi),\quad \phi(\xi) \ge c|\xi|,\quad \lambda\in \R,\quad \xi\in \R^2,
$$
for some $c>0.$ We let
 $$
 B_\phi:=\{\xi\in \R^2:\,\, \phi(\xi)\le1\}.
 $$
The dual of $\phi$ is defined as 
 $$
 \phi^o(\xi) = \max\limits_{\eta\in \R^2,\,\,\phi(\eta) = 1} \,\,\xi\cdot \eta,\quad \xi\in \R^2,
 $$
which turns out to be an anisotropy. Our convention is that $\phi$ measures $1$-vector fields and $\phi^o$ measures $1$-covectors fields (one-forms); so the
domain of $ \phi $ (resp. $ \phi^o $) is the tangent (resp. cotangent) space at a point of $ \R^2$. We do not use different symbols for the domain of $ \phi $ and $ \phi^o $.
Notice that $\phi^{oo} = \phi.$ 
$B_ \phi $ is sometimes called Wulff shape, and $B_{ \phi^o }$ Frank
diagram.

We say that $\phi$ is \emph{elliptic} if $\phi\in C^2( \R^2 \setminus \{0\})$ 
and
$$
\nabla^2\phi(\nu)\tau\cdot \tau \ge \bar c > 0
$$
for all $\tau,\nu\in\S^1$ with $\tau\cdot \nu=0$\footnote{If $\psi(\theta):=\phi(\cos\theta,\sin\theta)$, this inequality becomes $\psi + \psi''\ge \bar c$.}. One checks that if $\phi$ is elliptic 
then $\phi^o$ is also elliptic. It is well-known \cite[Chapter 1]{Cioranescu:} that $\phi \in C^1( \R^2 \setminus \{0\})$ 
if and only if $\phi^o$ is strictly convex\footnote{A function $f:\R^n\to \R$ is \emph{strictly convex} if for any $x,y\in \R^n$ and $\alpha\in(0,1)$ one has $f(\alpha x+ (1-\alpha) y) < \alpha f(x) + (1-\alpha)f(y).$}.
We say $\phi$ is \emph{crystalline} if $B_\phi$  is a convex polygon. 
It can be readily checked that $\phi$ is crystalline if and only if so is $\phi^o.$

In this paper we assume that an anisotropy and its dual are 
 either both elliptic or both crystalline. 
Even though some
notions that we are going to introduce hold also in other 
cases (for instance when $\phi$ is smooth but not strictly convex\footnote{We are not aware of any local 
existence result 
of network evolutions in these cases.}), and despite of
their interest, 
we shall not consider them here. 
 
\subsection{Curves}
A \emph{curve} in $ \R^2$ is the image of a continuous function $\sigma \in C^0([0,1]; \R^2).$  In this survey we consider only \emph{embedded} curves, i.e., 
with no self-intersections except the endpoints. If $\sigma(0) = \sigma(1),$ the curve is called \emph{closed}.
When $\sigma$ is $C^1$ (resp. Lipschitz) and $|\sigma'| > 0$ in $[0,1]$ (resp. a.e. in $[0,1]$), the map $\sigma$ is called a regular parametrization of $\Sigma:=\sigma([0,1]).$  A curve is called $C^{k+\alpha}$ for some $k\ge0$ and $\alpha\in[0,1],$ if it admits a regular $C^{k+\alpha}$-parametrization.
The tangent line to $\Sigma$ at its point $\pnt$ is denoted by $T_{\pnt}\Sigma.$ The (Euclidean) unit tangent vector to $\Sigma$ at $\pnt$ is denoted by $\tau_\Sigma(\pnt)$ and the unit normal vector is $\nu_\Sigma(\pnt) = \tau_\Sigma(\pnt)^\perp.$ Namely, if $\pnt = \sigma(x),$ then 
$$
\tau_\Sigma^{}(\pnt) = \frac{\sigma'(x)}{|\sigma'(x)|}
\qquad\text{and}\qquad 
\nu_\Sigma^{}(\pnt) = \frac{\sigma'(x)^\perp}{|\sigma'(x)|}.
$$

\begin{definition} 
Given $\pnt \in\p \Sigma$ and a nonzero vector $z\in  \R^2\setminus T_{\pnt}\Sigma$ (in case $T_{\pnt}\Sigma$ exists)  we write 
\begin{equation}\label{conormal_tonagentos}
z^{\p\Sigma}(\pnt) 
\end{equation}
to denote the $90^{\circ}$-rotation of $z$ pointing out of the curve.
\end{definition}

Sometimes we consider sets $\Sigma$ for which there exists $R_0>0$ such that $\Sigma\cap \cl{D_R}$ is a Lipschitz (resp. $C^{k+\alpha}$) curve with boundary and $\Sigma\setminus D_R$ is a straight half-line for any disc $D_R$ with $R>R_0$. With a slight abuse of notation, such sets $\Sigma$ will be still called a Lipschitz (resp. $C^{k+\alpha}$) curve with boundary. In this case $\Sigma$ has only one boundary point.

\subsection{Tangential divergence of a vector field}

The \emph{tangential divergence} of a vector field $g\in C^1( \R^2;  \R^2)$ 
over an embedded Lipschitz curve $\Sigma$ is defined as
$$
\div_\Sigma g(\pnt) = \nabla g(\pnt)\tau_\Sigma(\pnt)\cdot \tau_\Sigma(\pnt),\quad \text{$\cH^1$-a.e. $\pnt\in\Sigma$}.
$$
When the curve is $C^1$, this equality holds at every point of $\Sigma.$

\begin{remark}
When we will define Cahn-Hoffman vector fields, we consider the tangential divergence of a Lipschitz vector field $N_\phi$ defined only along a Lipschitz curve $\Sigma.$ In this case, we extend $N_\phi$ to a tubular neighborhood of $\Sigma$  constant along the vector $N_\phi(\pnt)$ for $\pnt\in\Sigma,$ i.e., if $z\in  \R^2$ and $z = \pnt + \lambda N_\phi(\pnt)$ for a unique $\pnt\in \Sigma$ and sufficiently small $|\lambda|,$ then we set $N_\phi(z): = N_\phi(\pnt).$ 
\end{remark}

The tangential divergence can also be introduced using parametrizations. More precisely, if $\sigma\in \Lip([0,1]; \R^2)$ is a regular parametrization of $\Sigma$ and $g:\Sigma\to  \R^2$ is a Lipschitz vector field along $\Sigma$, i.e., $g\circ \sigma\in \Lip([0,1]; \R^2)$, then 
\begin{equation}
\div_\Sigma\, g(\pnt) = \frac{[g\circ \sigma]'(x)\cdot \sigma'(x)}{|\sigma'(x)|^2},\qquad \pnt =\sigma(x)
\label{lalala_season3939}
\end{equation}
at points of differentiability.
One can readily check that the tangential divergence is independent of the parametrization.

\subsection{Lipschitz/smooth  partitions and associated networks}

Given a finite family $\{E_i\}$ of open subsets (the phases) of $ \R^2$ with Lipschitz boundary such that $\bigcup_i \cl{E_i} =  \R^2$ and $E_i\cap E_j = \emptyset$ for $i\ne j,$ we say $\{E_i\}$ is a (finite) \emph{Lipschitz} (resp. \emph{$C^{k+\alpha}$})  partition of $ \R^2$ if $\Sigma_{ij}:=\p E_i\cap \p E_j$ (if not empty or discrete, see Figure \ref{fig:lip_and_nonlip}(a)) is a finite union of Lipschitz (resp. $C^{k+\alpha}$) curves with boundary. Each $E_i$ is called a \emph{phase} and each $\Sigma_{ij}$ (if not discrete) is called an \emph{interface}. 
Given a natural number $m\ge3, $ we call a point $q$ an \emph{$m$-tuple junction} of $\{E_i\}$ is there exist (exactly) $m$-phases containing $q$ in their boundary.

\begin{figure}[]
\begin{center}
\includegraphics[width=0.8\textwidth]{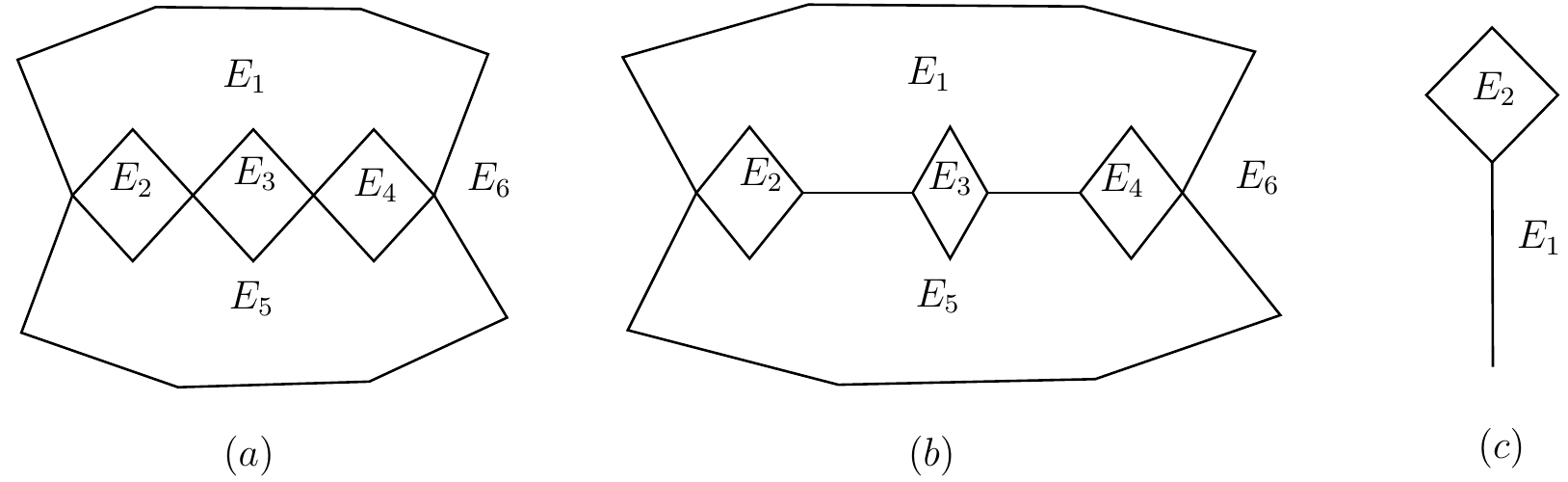}
\caption{\small Lipschitz (Figures (a) and (b)) and non Lipschitz (Figure (c)) partitions. Note that in (a) $\Sigma_{15}=\p E_1\cap \p E_5$ consists of four points and $\Sigma_{36}=\p E_3\cap \p E_6$ is empty. In (b) $\Sigma_{15}$ is a disjoint union of two Lipschitz curves and two points and in (c) the unbounded phase $E_1$ cannot be written (locally) as a subgraph of a Lipschitz function.}
\label{fig:lip_and_nonlip}
\end{center}
\end{figure}

In what follows we consider only Lipschitz partitions of $ \R^2$ for which:
\begin{itemize}
 \item $\Sigma:=\bigcup_{i,j}\Sigma_{ij}$ (which we call a \emph{network}) is connected,
 
 \item either only one phase is unbounded or $\Sigma$ consists of finitely many half-lines out of some discs (this case will be considered only in the crystalline case).
\end{itemize}
In particular, we do not prescribe Dirichlet boundary conditions for networks; moreover, in the evolutions we will admit only triple junctions. For notational simplicity, the curves of $\p E_i\cap \p E_j$ in the network $\Sigma$ will be often denoted by $\Sigma_k,$ using one index only. 

Note that the Lipschitzianity of $E_i$ imply that our Lipschitz partitions do not include Brakke's spoon\footnote{A union of an embedded closed curve and a half-line starting from a point of the curve.} 
type networks (in this case the unbounded phase is not a Lipschitz set, see Figure \ref{fig:lip_and_nonlip} (c)).

\subsection{Anisotropic energy of a network}

Let $\{E_i\}_{i=1}^n$ be a Lipschitz partition of $ \R^2.$ 
Let $\Phi:=\{\phi_{ij}\}$ be a collection of anisotropies in $ \R^2$ such that each $\phi_{ij}$ is associated to $\Sigma_{ij}.$ Notice that $\phi_{ij}=\phi_{ji}$ and $\Sigma_{ij} = \Sigma_{ji}.$ The  \emph{$\Phi$-length} of $\Sigma:=\bigcup_{ij}\Sigma_{ij}$ in an open set $U\subseteq \R^2$ is defined as 
\begin{equation}\label{network_length}
\ell_\Phi(\Sigma; U): = \sum\limits_{1\le i < j\le n} \,\int_{U\cap \Sigma_{ij}} \phi_{ij}^o(\nu_{\Sigma_{ij}})d\cH^1. 
\end{equation}
By assumption, each $\Sigma_{ij}$ is either empty, or a finite set of points or a Lipschitz curve with boundary and therefore, $\ell(\Sigma;U)<+\infty$ for any bounded open set $U.$ We also set 
$$
\ell_\Phi(\Sigma) : = \ell_\Phi(\Sigma;  \R^2)
$$
provided that $\Sigma$ is bounded.

\begin{remark}
We assume
\begin{equation}
\phi_{ij}^o + \phi_{jk}^o\ge \phi_{ik}^o,
\label{triangle_ineqeqeq}
\end{equation}
which is important since the invalidity of \eqref{triangle_ineqeqeq} yields local instabilities. Indeed, in this situation, a creation of a very thin new phase along the interfaces with large surface tensions would decrease the length. 
In the proof of Theorem \ref{teo:existence_special_mcf} we do not use \eqref{triangle_ineqeqeq} because of our assumptions on the shape of admissible networks (we do not allow creation of new phases).
\end{remark}

\section{Evolution of networks with elliptic anisotropies}

In this section we assume that all anisotropies are elliptic.

\subsection{First variation of length}

The following result 
was established in \cite[Theorem 3.4]{BNR:2003}.

\begin{proposition}\label{prop:first_varaition_curve}
Let $\Sigma\subset \R^2$ be an embedded  $C^2$ curve with boundary $\p\Sigma = \{p,q\}$ and let $\sigma:[0,1]\to  \R^2$ be a 
regular parametrization of $\Sigma$ 
with $\sigma(0)=p_0$ and $\sigma(1)=p_1$. Let $\beta\in C^2([0,1]; \R^2)$ and for 
sufficiently small $|s|$ with $s\in \R$, let $\sigma + s\beta$ parametrize the curve $\Sigma_s.$ Then 
\begin{equation*}
\frac{d \ell_\phi(\Sigma_s)}{ds}\Big|_{s=0} = 
\int_\Sigma  
\beta\cdot \nu_\Sigma\,\,\div_\Sigma N_\phi \,d\cH^1 + \beta(1)\cdot [N_\phi(p_1)]^{\p\Sigma} +  \beta(0)\cdot [N_\phi(p_0)]^{\p\Sigma},  
\end{equation*} 
where
$$
N_\phi(\pnt): = \nabla \phi^o(\nu_\Sigma(\pnt)),\quad \pnt \in\Sigma,
$$ 
and $N_\phi^{\p \Sigma}$ is defined in  \eqref{conormal_tonagentos}.
\end{proposition}

The number 
$$
\kappa_\Sigma^\phi(\pnt):= \div_\Sigma N_\phi(\pnt),\qquad \pnt\in\Sigma,
$$
in the integral is called the \emph{$\phi$-curvature}\footnote{In higher dimensions the anisotropic tangential divergence of a vector field $N:\Sigma\to \R^n$ over a Lipschitz manifold $\Sigma$ is defined as
$$
\div_{\Sigma,\phi}\, g = {\rm Tr}\Big(({\rm Id} - n_\phi\otimes \nu_\phi) \nabla \hat g\Big),
$$
where $\nu_\phi = \nu_\Sigma/\phi^o(\nu_\Sigma)$ and $n_\phi\in \p \phi^o(\nu_\phi),$ and $\hat g$ is the constant extension of $g$ along $n_\phi$ \cite[Definition 4.1]{BNP:2001}. By \cite[Lemma 4.4]{BNP:2001} $\div_{\Sigma,\phi} N_\phi$ coincides with $\div_\Sigma N_\phi.$
} of the curve $\Sigma$ at $\pnt$ and the vector field $N_\phi$ is sometimes called the \emph{Cahn-Hoffman vector field} on $\Sigma$. Moreover, the vector  $\kappa_\Sigma^\phi N_\phi$ is called the $\phi$-vector curvature. 
When no confusion arises, we write $\kappa^\phi$ in place of $\kappa_\Sigma^\phi$.

From Proposition \ref{prop:first_varaition_curve} we get 

\begin{corollary}
Let $\Sigma:=\{\Sigma_i\}_{i=1}^n$ be a network consisting of embedded $C^2$-curves with boundary and $\Phi=\{\phi_i\}_{i=1}^n$ be elliptic anisotropies such that $\phi_i$ is associated to $\Sigma_i.$ Let $Q$ be a $m$-tuple junction, say the intersection point of curves $\Sigma_1,\ldots,\Sigma_m.$ Let $\sigma_i:[0,1]\to \R^2$ be a regular $C^2$-parametrization of $\Sigma_i$ such that $\sigma_i(1) = Q$ and let $\beta_1,\ldots,\beta_m\in C^2([0,1]; \R^2)$ be such that $\beta_i(0)=0$ and $\beta_i(1)=\beta_j(1)=Q$ and for $s\in \R$ with 
sufficiently small $|s|$ let $\Sigma_i^s$ be the curve parametrized by $\sigma_i+s\beta_i$ and set $\Sigma_s:=\bigcup_{i\ge1} \Sigma_i^s.$ 
Then 
\begin{equation}\label{first_var_networkos}
\frac{d\ell_\Phi(\Sigma_s)}{ds}\Big|_{s=0} = \sum\limits_{i=1}^m \int_{\Sigma_i}  \beta_i\cdot \nu_{\Sigma_i} \,\kappa^{\phi_i}\, d\cH^1 + Q\cdot \sum\limits_{i=1}^m [N_{\phi_i}(Q)]^{\p\Sigma_i}. 
\end{equation}
\end{corollary}

The balance condition 
\begin{equation}\label{herring_condosh}
\sum\limits_{i=1}^m [N_{\phi_i}(Q)]^{\p\Sigma_i}  = 0
\end{equation}
is sometimes called  \emph{Herring condition} \cite{BCN:2006,GN:2000,Herring:1999,TCH:1992}. By the definition \eqref{conormal_tonagentos} of $z^{\p\Sigma}$ this equality is rewritten also as 
\begin{equation}\label{herring_condosh_normals}
\sum\limits_{i=1}^m N_{\phi_i}(Q)  = 0.
\end{equation}

Condition \eqref{herring_condosh} requires some 
compatibility between anisotropies $\phi_i:$ 

\begin{example}
Let $n = 3,$ $\phi_1 = \phi_2=|\cdot|$ and $\phi_3=1/3|\cdot|,$ and let $Q$ be a triple junction of $C^2$-curves $\Sigma_1,$ $\Sigma_2$ and $\Sigma_3.$ Then
$$
N_1(Q)= \nu_{\Sigma_1}(Q),\quad N_2(Q) = \nu_{\Sigma_2}(Q)\quad\text{and}\quad N_3(Q) = 3\nu_{\Sigma_3}(Q).
$$
Thus, $|N_1| = |N_2|=1$ and $|N_3|=3,$ which implies the sum 
$N_1(Q) + N_2(Q) + N_3(Q)$ can not be zero. In particular, condition \eqref{herring_condosh_normals} is not related to the triangle inequality \eqref{triangle_ineqeqeq}.
\end{example}

\subsection{Anisotropic  curvature of a curve}

Let $\Sigma$ be an embedded  $C^2$-curve regularly parametrized by $\sigma\in C^2([0,1]; \R^2).$ Let us express $\kappa^\phi$ by means of $\sigma.$ Note that
$$
N_\phi(\pnt) = \nabla \phi^o\Big(\tfrac{\sigma'(x)^\perp}{|\sigma' (x)|}\Big),\quad \pnt=\sigma(x)\in\Sigma, 
$$
and hence, by \eqref{lalala_season3939} at $\pnt = \sigma(x)$ we have
\begin{equation}
\kappa^\phi(\pnt) = \left(\Big[\nabla^2\phi^o\Big(\tfrac{\sigma' (\pnt)^\perp}{|\sigma' (\pnt)|}\Big)\tfrac{\sigma' (\pnt)}{|\sigma' (\pnt)|}\Big]\cdot \tfrac{\sigma' (\pnt)}{|\sigma' (\pnt)|}\right)\,\tfrac{\sigma''(\pnt)\cdot \sigma' (\pnt)^\perp}{|\sigma' (\pnt)|^3}.
\label{sakdakda}
\end{equation}
Now recalling the definition of $\tau_\Sigma$ and $\nu_\Sigma$ as well as the definition of the Euclidean curvature $\kappa$ of a curve, the last equality is rewritten on $\Sigma$ as 
$$
\kappa^\phi = \Big(\nabla^2\phi^o(\nu_\Sigma)\tau_\Sigma\cdot \tau_\Sigma\Big)\,\kappa.
$$
This observation will be used frequently.

\subsection{Existence of a smooth flow}

In this section we only consider bounded networks associated to an at least  $C^2$-partition of $ \R^2$ and all anisotropies are at least $C^3.$ 

\begin{definition} 
Given a network $\Sigma^0:=\{\Sigma_j^0\}_{j=1}^n$  and associated elliptic anisotropies\footnote{For simplicity, we write $\phi_k$ and $\Sigma_k$ in places of $\phi_{ij}$ and $\Sigma_{ij}.$} $\Phi:=\{\phi_j\}_{j=1}^n,$  we say that a family $\Sigma(t): = \{\Sigma_j(t)\},$ $t\in[0,T)$ of networks is a \emph{$\Phi$-curvature flow} starting from $\Sigma^0$ if $\Sigma(0) = \Sigma^0$ and there exists an $n$-tuple $u:= \{u^j\}_{j=1}^n \in [C^{1,2}((0,T)\times [0,1]; \R^2)]^n\cap [C([0,T)\times [0,1]; \R^2)]^n$ such that
	
\begin{itemize}
\item[(a)] each $u^j(t),$ $t\ge0,$ is a regular parametrization of $\Sigma^j(t);$

\item[(b)]  
\begin{equation}\label{curves_clow_mcs}
u_t^j\cdot \nu^j = \phi_j^o(\nu^j)\big(\nabla^2\phi_j^o(\nu^j)\tau^j\cdot \tau^j\big)\kappa \quad \text{on $(0,T)\times (0,1)$}, 
\end{equation} 
where $\nu^j= \nu_{\Sigma^j}$ and $\tau^j= \tau_{\Sigma^j};$

\item[(c)] each $\Sigma(t)$ contains only triple junctions and if the curves $\Sigma^{j_1}(t),$ $\Sigma^{j_2}(t)$ and $\Sigma^{j_3}(t)$ intersect at a triple    junction $q(t),$ then
$$
\sum\limits_{i=1}^3 \nabla\phi_{j_i}^o(\nu^{j_i}(q(t))) = 0\quad \text{for all  $t\in (0,T).$}
$$
\end{itemize}
Any such flow $\Sigma(\cdot)$ is called a \emph{
smooth geometric anisotropic curvature flow of the network} $\Sigma^0$. 
\end{definition}

Condition (b) says that each curve in the network moves with normal velocity equal to its anisotropic curvature, whereas condition (c) expresses the Herring condition at  triple junctions. 
Note that we are not assuming a priori that $\Sigma^0$ satisfies Herring condition \eqref{herring_condosh_normals}. Even though in what follows we consider only initial networks satisfying \eqref{herring_condosh_normals}, we would like to mention that  there are results in the Euclidean case (see e.g. \cite{LMPS:2021}) that prove short time existence from an initial network not satisfying \eqref{herring_condosh_normals}; this instantaneous regularization is an interesting result.

Equation \eqref{curves_clow_mcs} expresses only the normal component of the velocity of $\Sigma(t);$ the presence of triple junctions forces (see e.g. \cite{KNP:2021,MNPS:2016,MNT:2004}) $\Sigma(t)$ to have also a tangential velocity. Following \cite[Definition 2.4]{KNP:2021} and choosing the tangential component of the velocity as 
$$
\lambda:=u_t^j\cdot \tau^j = \phi_j^o(\nu^j)(\nabla^2 \phi_j^o(\nu)\tau^j\cdot \tau^j) \tfrac{u_{xx}^j}{|u_x^j|^2} \cdot \tau^j
$$
in \eqref{curves_clow_mcs},  we can introduce:

\begin{definition}[\textbf{Special geometric flow}]
A \emph{special geometric anisotropic curvature flow} is defined by the equation 
\begin{equation}\label{special_flows}
u_t^j = \phi_j^o(\nu^j)(\nabla^2 \phi_j^o(\nu)\tau^j\cdot \tau^j) \frac{u_{xx}^j}{|u_x^j|^2}.
\end{equation}
\end{definition}

\begin{remark}[\textbf{Reduction to a special flow}]\label{rem:cjhalpal}
Repeating the arguments of \cite[Lemma 4.1]{KNP:2021} (see also \cite{MNPS:2016}) we can prove that using (orientation preserving) diffeomorphisms/reparametrizations every smooth geometric flow can be reduced to a special geometric flow. This observation implies that given $\Sigma^0$ (satisfying condition \eqref{herring_condosh_normals}), to prove the short-time existence of a smooth geometric flow we only need to establish short-time existence of a special geometric flow starting from $\Sigma^0.$ 
\end{remark}

Let us consider a special geometric flow $u(t)$ and the evolution of some triple junction 
$$
q(t):=u^{j_1}(t,y_1) =u^{j_2}(t,y_2)=u^{j_3}(t,y_3)
$$ 
for some $(y_1,y_2,y_3)\in\{0,1\}^3$ and for all $t\in(0,1)$. Since  all $u^j\in C^{1,2}((0,T)\times [0,1]; \R^2),$  we have 
\begin{equation}
q_t(t)= u_t^{j_1}(t,y_1) =u_t^{j_2}(t,y_2)=u_t^{j_3}(t,y_3).
\label{dahu3fg}
\end{equation}
Thus, inserting \eqref{sakdakda} in \eqref{special_flows} at $q(t)$ we get
\begin{align}
\phi_i^o(\nu^i)\big(\nabla^2\phi_i^o(\nu^i)\tau^i\cdot \tau^i\big)\,\frac{u_{xx}^i(t,y)}{|u_x^i(t,y)|^3} 
= 
\phi_j^o(\nu^j)\big(\nabla^2\phi_j^o(\nu^j)\tau^j\cdot \tau^j\big)\,\frac{u_{xx}^j(t,y)}{|u_x^j(t,y)|^3} 
\label{no_tangency}
\end{align}
for all $i,j\in\{j_1,j_2,j_3\}.$ If $\Sigma(\cdot)$ is smooth up to $t=0,$ the second order compatibility condition \eqref{no_tangency} should be satisfied by the initial network $\Sigma^0.$
Later in this section we show that if $\Sigma^0$ satisfies conditions
\eqref{herring_condosh_normals} 
and \eqref{no_tangency},  then there exist $T>0$ and a special geometric anisotropic curvature flow $\{\Sigma(t)\}_{t\in[0,T]}$ starting from $\Sigma^0.$

\subsection{Role of the Herring condition}

Let $\Sigma(t)$ be a smooth geometric flow starting from a bounded network $\Sigma^0.$ We claim that condition \eqref{herring_condosh_normals} 
implies that the anisotropic length $\ell_\Phi(\Sigma(t))$ is non-increasing in $t>0.$ Indeed, without loss of generality, we may assume that $\Sigma(t)$ is special (see Remark \ref{rem:cjhalpal}).  Hence, using the definition of $\ell_\Phi,$ integration by parts, \eqref{dahu3fg} and \eqref{curves_clow_mcs} we get 
\begin{align*}
\frac{d}{d t}\,\ell_\Phi(\Sigma(t)) = - \sum\limits_{j=1}^n 
\int_0^1 \frac{\phi_j^o(u_x^\perp)}{|u_x^j|^3}\,\Big|\nabla^2 \phi_j^o(u_x^\perp)u_{xx}^j\Big|^2\,d x 
+ 
\sum\limits_{q(t)\in J(t)} q_t(t)\cdot \sum\limits_{i=1}^3 \nabla\phi_{j_i}^o(\nu^{j_i}(q(t))),
\end{align*}
where $J(t)$ is the set of all triple junctions, 
$$
q(t): = u^{j_1}(t,y_1) =u^{j_2}(t,y_2)=u^{j_3}(t,y_3)
$$
($j_i$ and $y_i$ are different for different $q$).
Hence, condition \eqref{herring_condosh_normals} 
implies that 
$$
\frac{d}{d t}\,\ell_\Phi(\Sigma(t))  \le 0,
$$
i.e. the map $t\mapsto \ell_\Phi(\Sigma(t))$ is non-increasing.

\subsection{Existence and uniqueness of special flows}

In this section we prove the short-time existence and uniqueness of a smooth anisotropic curvature flow  starting from a given network $\Sigma^0$ satisfying 
conditions \eqref{herring_condosh_normals} and \eqref{no_tangency} at triple junctions. For simplicity, we consider only \emph{theta-shaped} networks, i.e. bounded $C^2$-networks consisting of only three embedded curves meeting at two triple junctions.

The main result of the section reads as follows.

\begin{theorem}[Local existence and uniqueness] 
Let $\Sigma^0$ be a $C^{2+\alpha}$ theta-shaped network satisfying at both triple junctions
\begin{equation*} 
\sum\limits_{i=1}^3 \nabla\phi_i^o(\nu_{\Sigma^0}^{i}) = 0 
\end{equation*}
and admitting a parametrization satisfying the second order compatibility condition  \eqref{no_tangency}. Then there exists a unique smooth geometric flow starting from $\Sigma^0.$
\end{theorem}

As observed in Remark \ref{rem:cjhalpal}, concerning existence we just need to prove the existence of a special flow. Then  uniqueness follows from  uniqueness of the special flow.

We postpone the proof after several auxiliary results. Before going further, we recall some notions related to parabolic H\"older spaces. 
For a function $v: [0,T] \times [0,1] \to \R$ and $\alpha\in(0,1]$ we let 
$$
[v]_{\alpha,x} : = \sup\limits_{t\in[0,T],\,x,y\in[0,1],\,x\ne y}\,\,\frac{|v(t,x) - v(t,y)|}{|x-y|^\alpha },
$$
$$
[v]_{\alpha,t} : = \sup\limits_{s,t\in[0,T],\,s\ne t,\, x\in[0,1]}\,\,\frac{|v(s,x) - v(t,x)|}{|s-t|^\alpha }.
$$
For $\alpha\in(0,1]$ and $k\in\N_0:=\N\cup \{0\}$ we denote by $C^{\frac{k + \alpha}{2}, k+\alpha}([0,T] \times [0,1])$ the space of all functions $v:[0,T]\times[0,1]\to \R$ whose continuous derivatives $\p_t^i\p_x^j v$ exist for all $i,j\in\N_0$ with $2i + j\le k$ and satisfy
\begin{multline*}
\|v\|_{C^{\frac{k+\alpha}{2},k+\alpha}([0,T]\times[0,1])}:=  \sum\limits_{2i+j = 0}^k \sup\limits_{t\in[0,T],\,x\in[0,1]}\,\,|\p_t^i\p_x^j v(t,x)|\\
+ \sum\limits_{2i+j = k} [\p_t^i\p_x^j v]_{\alpha,x} + 
\sum\limits_{k + \alpha - 2i-j<2}\,\,[\p_t^i\p_x^j v]_{\frac{k + \alpha - 2i - j}{2},t} < +\infty.
\end{multline*}
The $C^{\frac{k+\alpha}{2},k+\alpha}$-norm of a \emph{vector valued map} is the sum of the norms of its components.
We also adopt the following conventions:
\begin{itemize}
\item whenever it is clear from the context, we set 
$$
C_T^{\frac{k+\alpha}{2},k+\alpha}:= C^{\frac{k+\alpha}{2},k+\alpha}([0,T] \times [0,1]; \R^2)
$$
and
$$
\|v\|_{C_T^{\frac{k+\alpha}{2},k+\alpha}}: = \|v\|_{C^{\frac{k+\alpha}{2},k+\alpha}([0,T]\times[0,1]; \R^2)};
$$

\item for functions $v$ depending only on one variable (space or time), we set
$$
\holder^{k+\alpha}:= C^{k+\alpha}([0,1]; \R^2)
\quad \text{and}\quad 
\holder_T^{\frac{k+\alpha}{2}}:= C^{\frac{k+\alpha}{2}}([0,T]; \R^2) 
$$
and 
$$
\|v\|_{k+\alpha}:= \|v\|_{C^{k+\alpha}([0,1]; \R^2)}
\quad \text{and}\quad 
\|v\|_{\frac{k+\alpha}{2},T}:= \|v\|_{ C^{\frac{k+\alpha}{2}}([0,T]; \R^2) }.
$$
\end{itemize}
By 
$$
\networks_n^{k,\alpha} 
$$
we denote the set of $n$-tuples $\sigma:=(\sigma^1,\ldots,\sigma^n)\in [\holder^{k+\alpha}]^n$ such that $\bigcup_{i = 1}^n \sigma^i([0,1])$ is a network with only triple junctions.
Similarly, we denote by
$$
\networks_n^{k,\alpha, T} 
$$
the set of all $n$-tuples $\sigma:=(\sigma^1,\ldots,\sigma^n)\in \big[\holder_T^{\frac{k+\alpha}{2},k+\alpha}\big]^n$ such that $\sigma(t) = (\sigma^1(t,\cdot),\ldots,\sigma^n(t,\cdot)) \in \networks_n^{k,\alpha}$ for any $t\in[0,T]$.

We start with a general result related to the existence of special smooth geometric flows. 

\begin{assumption}\label{assumptions}
$\,$
\begin{itemize}
\item[(A)] $\Theta:=(\theta_i,\theta_2,\theta_3)$ are positively one-homogeneous $C^{3+\alpha}$-functions in  $ \R^2\setminus\{0\}$ for some $\alpha\in(0,1],$

\item[(B)] $\cB:=(\beta_1,\beta_2,\beta_3)$ are even $C^{2+\alpha}$-functions defined in a tubular neighborhood of the unit circle $\S^1$ such that  
\begin{equation*} 
0<m \le \min_{\nu\in \S^1} \beta_i(\nu) \le \max_{\nu\in \S^1} \beta_i(\nu)  \le \frac{1}{m},\quad i=1,2,3, 
\end{equation*}
for some $m\in(0,1],$

\item[(C)] $\cF:=(f^1,f^2,f^3)\in (\holder_T^{\alpha})^3.$
\end{itemize}
\end{assumption}

The first variation formula \eqref{first_var_networkos} for length shows that for the anisotropic curvature flow  we need to choose 
$$
\theta_i(\tau) = \phi_i^o(\tau^\perp)\quad\text{and}\quad \beta_i(\tau) = \phi_i^o(\tau^\perp) \nabla^2\phi_i^o(\tau^\perp)\tau\cdot\tau\quad \text{and}\quad f_i=0.
$$ 
Notice that in Assumption \ref{assumptions} (A) we are not assuming $\theta_i$ to be even.

\begin{theorem}\label{teo:existence_special_mcf}
Let $k\ge2,$ $\alpha\in(0,1]$ and let $\sigma \in \networks_3^{k,\alpha}$ satisfy the compatibility conditions 
\begin{align}\label{init_condition}
\begin{cases}
\sigma^1(y) =  \sigma^2(y) =\sigma^3(y), \\
\sum\limits_{i=1}^3 \nabla \theta_i\Big(\tfrac{\sigma_x^i(y)}{|\sigma_x^i(y)|}\Big) = 0, \\
\beta_i\Big(\tfrac{\sigma_x^i(y)}{|\sigma_x^i(y)|}\Big)\, \tfrac{\sigma_{xx}^i(y)}{|\sigma_x^i(y)|^2}
= \beta_j\Big(\tfrac{\sigma_x^j(y)}{|\sigma_x^j(y)|}\Big)\, \tfrac{\sigma_{xx}^j(y)}{|\sigma_x^j(y)|^2},\quad i,j=1,2,3, 
\end{cases}
\end{align}
 whenever  $y = 0$ or $y=1$.
Let $\cF$ be such that 
\begin{equation*}
f^i(t,y) = f^j(t,y)\qquad \text{whenever $(t,y)\in[0,T]\times \{0,1\}.$} 
\end{equation*}
Then there exist  $T>0$ and a unique flow $u(t) = (u^1(t,\cdot),u^2(t,\cdot),u^3(t,\cdot))\in \networks_{3}^{k+2,\alpha,T}$ of networks such that  
\begin{align}\label{existed_flow}
\begin{cases}
u(0,x) = \sigma(x) & \text{for $x\in[0,1],$}\\
u^1(t,y) =  u^2(t,y) =u^3(t,y)  & \text{for $t\in [0,T]$ and $y\in\{0,1\},$}\\
\sum\limits_{i=1}^3\nabla  \theta_i\Big(\tfrac{u_x^i(t,y)}{|u_x^i(t,y)|}\Big) = 0, & \text{for $t\in [0,T]$ and $y\in\{0,1\},$}\\
u_t^i = \beta_i\Big(\tfrac{u_x^i}{|u_x^i|}\Big)\, \tfrac{u_{xx}^i}{|u_x^i|^2} + f^i & \text{in $[0,T]\times[0,1]$ and $i=1,2,3.$} 
\end{cases}
\end{align} 
\end{theorem}

To prove the theorem we follow the arguments of \cite[Section 3]{KNP:2021}. Namely, we linearize the problem near the initial and boundary data, then using the results of Solonnikov \cite{Sol:1965} we solve the linear problem, and finally, using careful H\"older estimates we reduce the problem to a Banach fixed point argument. Note that in \cite{KNP:2021} the authors consider just one anisotropy  and networks having a single triple junction together with a Dirichlet boundary condition. 

We divide the proof into several steps.

\subsubsection{Main functional spaces} For $T>0,$ $k\ge2,$ $\alpha\in(0,1]$ and $M>0$ let 
\begin{equation*}
X_{M,T}^j:= \Big\{v\in \holder_T^{\frac{k+\alpha}{2},k+\alpha}:\,\, \|v\|_{C_T^{\frac{k+\alpha}{2},k+\alpha}}\le M,\,\,v(0,\cdot) = \sigma^j(\cdot)  \Big\} 
\end{equation*}
for $j =1,2,3.$ Note that if $M$ is large, then $X_{M,T}^j$ is non-empty. 

\begin{lemma}\label{lem:regular_parametirze}
For $\sigma \in \networks_3^{k,\alpha}$ assume that 
$$
\min\limits_{z\in[0,1]} |\sigma_x^j(z)| = \delta>0,\quad j=1,2,3.
$$
Then for any $v\in X_{M,T}^j$ one has
$$
\min\limits_{t\in[0,T],\,z\in[0,1]} |v_x(t,z)| \ge \delta/2 
$$
provided that 
\begin{equation}\label{maxT}
T\le \frac{\delta}{2M}. 
\end{equation}
\end{lemma}

\begin{proof}
Since
$$
v_x(t,z) = v_x(0,z) + \int_0^t v_{tx}(s,z)ds,
$$
by the H\"older estimate of $v\in X_{M,T}^j$ we have 
$$
|v_x(t,z)| \ge |v_x(0,z)| - \int_0^t |v_{tx}(s,z)|ds \ge \delta - Mt \ge \frac{\delta}{2},
$$
whenever $t\le T\le \frac{\delta}{2M}.$
\end{proof}

From now on we assume \eqref{maxT} and we set 
\begin{equation}\label{main_func_space1}
\cR_{M,T}: =  \prod_{j=1}^3 X_{M,T}^j. 
\end{equation}
This will be our main functional space.

\subsubsection{Linearized problem}

Given $\bar u=(\bar u_1,\bar u_2,\bar u_3)\in\cR_{M,T},$ we look for $u = (u_1,u_2,u_3)\in \big[\holder_T^{\frac{k+\alpha}{2},k+\alpha}\big]^3$ solving the linear system of PDE's
\begin{equation}\label{eq:parPDE}
\begin{cases} 
& u_t^j - \alpha_j u_{xx}^j = F_{\bar u}^j,\\
& u^j(0,x) = \sigma^j(x), 
\end{cases}
\end{equation}
where $j=1,2,3,$
$$
\alpha_j: = \frac{\beta_j(\frac{\sigma_x^j}{|\sigma_x^j|})}{|\sigma_x^j|^2},\qquad 
F_{\bar u}^j:= f^j + 
\Bigg(\frac{\beta_j(\frac{\bar u_x^j}{|\bar u_x^j|})}{|\bar u_x^j|^2} -
\alpha_j\Bigg)\bar u_{xx}^j,
$$
with the linearized boundary conditions
\begin{align}
& u^1(t,y) =  u^2(t,y) =u^3(t,y), \label{triple_cond_lin}\\
& \sum\limits_{j=1}^3 \Bigg(\theta_j\Big(\tfrac{\sigma_x^j(y)}{|\sigma_x^j(y)|}\Big)\,\tfrac{u_x^j (t,y)}{|\sigma_x^j(y)|} + \bigg[\nabla \theta_j\Big(\tfrac{\sigma_x^j(y)}{|\sigma_x^j(y)|}\Big)\cdot \tfrac{[\sigma_x^j(y)]^\perp}{|\sigma_x^j(y)|}\bigg]\tfrac{[u_x^j(t,y)]^\perp}{|\sigma_x^j(y)|}\Bigg) = B_{\bar u}(t,y) \label{bound_cond_lin}
\end{align}
for $(t,y)\in[0,T]\times\{0,1\},$ where 
\begin{multline*}
B_{\bar u}(t,y):= \sum\limits_{j=1}^3 \Bigg\{ \Big[\theta_j\Big(\tfrac{\sigma_x^j(y)}{|\sigma_x^j(y)|}
\Big)\,\tfrac{1}{|\sigma_x^j(y)|} 
-
\theta_j\Big(\tfrac{\bar u_x^j(t,y)}{|\bar u_x^j(t,y)|}
\Big)\,\tfrac{1}{|\bar u_x^j(t,y)|} 
\Big]\, \bar u_x^j (t,y)\\
+ \bigg(\Big[\nabla \theta_j\Big(\tfrac{\sigma_x^j(y)}{|\sigma_x^j(y)|}\Big)\cdot \tfrac{[\sigma_x^j(y)]^\perp}{|\sigma_x^j(y)|}\Big]\tfrac{1}{|\sigma_x^j(y)|} 
-
\Big[\nabla \theta_j\Big(\tfrac{\bar u_x^j(t,y)}{|\bar u_x^j(t,y)|}\Big)\cdot \tfrac{[\bar u_x^j(t,y)]^\perp}{|\bar u_x^j(t,y)|}\Big]\tfrac{1}{|\bar u_x^j(t,y)|}\bigg)
[\bar u_x^j(t,y)]^\perp
\Bigg\}.
\end{multline*}
In the linearization \eqref{bound_cond_lin} of 
condition \eqref{herring_condosh_normals} 
we used the positive one-homogeneity and continuous differentiability of $\theta_j,$  hence,
$$
\theta_j(a) = \nabla \theta_j(a)\cdot a,\quad a\in \R^2\setminus\{0\}, 
$$
which implies 
$$
\nabla \theta_j(a) = (\nabla \theta_j(a)\cdot a)\, a + (\nabla \theta_j(a)\cdot a^\perp)\,a^\perp 
=\theta_j(a)\,a + (\nabla \theta_j(a)\cdot a^\perp)\,a^\perp,\quad a\in\S^1.
$$

\subsubsection{Solvability of the linear problem}
To check solvability we need to check that the linear system is compatible with boundary and initial data \cite{Sol:1965}.

We use the Fourier symbols $p=\p_t$ and $\xi = \p_x.$ 
The linear operator corresponding to the linear system \eqref{eq:parPDE} has the $6\times 6$-matrix 
$$
\cL(x,p,\xi) 
=
\begin{pmatrix}
(p - \alpha_1 \xi^2)\,\bI & \bO & \bO\\
\bO & (p - \alpha_2 \xi^2)\,\bI & \bO\\
\bO & \bO & (p - \alpha_3 \xi^2)\,\bI 
\end{pmatrix},
$$ 
where $\alpha_i:=\alpha_i(x),$ $\xi,p\in\C$ and 
$$
\bI = 
\begin{pmatrix}
1 & 0\\
0 & 1
\end{pmatrix}
,
\qquad
\bO = 
\begin{pmatrix}
0 & 0\\
0 & 0
\end{pmatrix}
.
$$
In particular, for $\ri = \sqrt{-1}$ and $\xi,p\in\C$
$$
L(x,p,\ri\xi) := \det \cL(x,p,\ri\xi) = \prod_{j=1}^3   (p + \alpha_j\xi^2)^2,
$$
and the matrix 
$$
\hat \cL(x,p,\ri\xi):= L(x,p,\ri\xi)\, \cL(x, p, \ri \xi)^{-1} 
$$ 
reads as 
\begin{multline*}
\hat \cL(x,p,\ri\xi) 
=\\
\begin{pmatrix}
(p + \alpha_2 \xi^2)(p + \alpha_3 \xi^2)\,\bI & \bO & \bO\\
\bO & (p + \alpha_1 \xi^2)(p + \alpha_3 \xi^2)\,\bI & \bO\\
\bO & \bO & (p + \alpha_1 \xi^2)(p + \alpha_2 \xi^2)\,\bI 
\end{pmatrix}.
\end{multline*} 
Since 
$$
\alpha_j = \frac{\beta_j(\sigma_x^j / |\sigma_x^j|)}{|\sigma_x^j|^2} \ge m\,\min\Big\{\tfrac{1}{|\sigma_x^j(x)|}:\,\,j=1,2,3,\,\,x\in[0,1]\Big\} > 0,
$$
the system \eqref{eq:parPDE} is parabolic, here $m$ is given by Assumption \ref{assumptions} (B).

Following \cite{KNP:2021},
fix $p\in\C\setminus\{0\}$ with $ \Re(p)\ge0$. Then  the polynomial $L(x,p,\ri\tau),$ $\tau\in\C,$ has six roots with positive imaginary part and six roots with negative imaginary part. More precisely, setting $p = |p|e^{\ri \zeta_p}$ with $|\zeta_p|\le \pi/2$ and
\begin{align}
\tau_j^+:=\tau_j^+(x,p) = & \sqrt{\frac{|p|}{\alpha_j}}\,e^{\ri \big(\tfrac\pi2 + \tfrac{\zeta_p}{2}
\big)}, \label{def:tau_j_plus}\\
\tau_j^-:=\tau_j^-(x,p) = & \sqrt{\frac{|p|}{\alpha_j}}\,e^{\ri \big(\tfrac{3\pi}{2} + \tfrac{\zeta_p}{2}
\big)}, \nonumber
\end{align}
we may write 
$$
L(x,p,\ri\tau) = \prod\limits_{j=1}^3 \alpha_j^2 (\tau - \tau_j^+)^2(\tau - \tau_j^-)^2.
$$
Let
$$
P^+:= P^+(x,p,\tau) = \prod\limits_{j=1}^3 (\tau - \tau_j^+)^2.
$$
Now we turn to define the matrix associated to the boundary conditions \eqref{triple_cond_lin}-\eqref{bound_cond_lin}. It is given as 
$$
\cB(y,\xi) = 
\begin{pmatrix}
\bI & -\bI & \bO\\ 
\bO & \bI & -\bI\\
\bQ_1 & \bQ_2 & \bQ_3
\end{pmatrix}
,\qquad y=0,1,\quad \xi\in\C,
$$
where
$$
\bQ_j: = b_{j1}\,
\begin{pmatrix}
\xi & 0\\
0 & \xi
\end{pmatrix}
+
b_{j2}\,
\begin{pmatrix}
0 & -\xi\\
\xi & 0
\end{pmatrix},\quad j=1,2,3,
$$
with all the coefficients  evaluated at $y = 0$ and $y = 1,$ and 
\begin{equation}
b_{j1}:= \tfrac{1}{|\sigma_x^j(y)|}\, \theta_j\Big(\tfrac{\sigma_x^j(y)}{|\sigma_x^j(y)|}\Big),\qquad 
b_{j2} :=   \nabla \theta_j\Big(\tfrac{\sigma_x^j(y)}{|\sigma_x^j(y)|}\Big)\cdot \tfrac{[\sigma_x^j(y)]^\perp}{|\sigma_x^j(y)|^2}.
\label{def:b_ij}
\end{equation}

Consider the matrix 
\begin{multline*}
\cA(y,p,\ri\tau) := 
\cB(y,\ri\tau) \hat \cL(y,p,\ri\tau)
= \\
\begin{pmatrix}
(p + \alpha_2 \tau^2)(p + \alpha_3 \tau^2)\,\bI & -(p + \alpha_1 \tau^2)(p + \alpha_3 \tau^2)\,\bI & \bO\\
\bO & (p + \alpha_1 \tau^2)(p + \alpha_3 \tau^2)\,\bI & -(p + \alpha_1 \tau^2)(p + \alpha_2 \tau^2)\,\bI\\
(p + \alpha_2 \tau^2)(p + \alpha_3 \tau^2)\,\bQ_1 & (p + \alpha_1 \tau^2)(p + \alpha_3 \tau^2)\,\bQ_2 & 
(p + \alpha_1 \tau^2)(p + \alpha_2 \tau^2)\,\bQ_3
\end{pmatrix}
.
\end{multline*}
By definition, the complementary condition holds \cite{Sol:1965} if the rows of this matrix are linearly independent modulo $P^+$ whenever $p\ne0$ with $\Re( p )\ge0.$
Thus, we need to check if $w\in \R^6$
(considered as a $1\times6$-matrix) is such that
$$
w\cdot \cA(y,p,\ri\tau) = (0,0,0,0,0,0)\quad\mod\,P^+,
$$ 
then $w = 0.$ This equation yields six linear equations. For instance, for the first column of $\cA$ we have 
$$
(p + \alpha_2 \tau^2)(p + \alpha_3 \tau^2)\,(w_1 + \ri b_{11} w_5\tau  + \ri b_{12} w_6\tau) = 0 \quad  \mod P^+.
$$
Then by the definition of $P^+$ we have 
$$
w_1 + \ri b_{11} w_5\tau  + \ri b_{12} w_6\tau= 0\quad \mod\,(\tau - \tau_1^+),
$$
or equivalently,
$$
w_1 + \ri b_{11} w_5\tau_1^+  + \ri b_{12} w_6\tau_1^+ = 0.
$$
Treating similarly the remaining columns we get 
\begin{equation*}
\begin{cases}
 w_1 + \ri b_{11} \tau^+_1 \, w_5 + \ri b_{12} \tau^+_1 \, w_6= 0,\\
 w_2 - \ri b_{12}\tau^+_1\, w_5 +  \ri b_{11} \tau^+_1\, w_6 = 0,\\
-w_1 + w_3 + \ri b_{21}\tau^+_2\,w_5 + \ri b_{22}\tau^+_2\,w_6 = 0,\\
-w_2 + w_4 - \ri b_{22}\tau^+_2\,w_5 + \ri b_{21}\tau^+_2\,w_6 = 0,\\
- w_3 + \ri b_{31} \tau^+_3 \, w_5 + \ri b_{32} \tau^+_3 \, w_6= 0,\\
 -w_4 - \ri b_{32}\tau^+_3\, w_5 +  \ri b_{31} \tau^+_3\, w_6 = 0.
\end{cases}
\end{equation*}
The determinant of this system is computed as 
\begin{align*}
\Delta:= & \det 
\begin{pmatrix}
1 & 0 & 0 & 0 & \ri b_{11}\tau_1^+ & \ri b_{12}\tau_1^+ \\
0 & 1 & 0 & 0 & - \ri b_{12}\tau_1^+ & \ri b_{11}\tau_1^+ \\
-1 & 0 & 1 & 0 & \ri b_{21}\tau_2^+ & \ri b_{22}\tau_2^+ \\
0 & -1 & 0 & 1 & -\ri b_{22}\tau_2^+ & \ri b_{21}\tau_2^+ \\
0 & 0 & -1 & 0 & \ri b_{31}\tau_3^+ & \ri b_{32}\tau_3^+ \\
0 & 0 & 0 & -1 & -\ri b_{32}\tau_3^+ & \ri b_{31}\tau_3^+
\end{pmatrix}
\\
= & -(b_{11} \tau_1^+ + b_{21} \tau_2^+ +b_{31} \tau_3^+) ^2 -
(b_{12} \tau_1^+ + b_{22} \tau_2^+ +b_{32} \tau_3^+) ^2.
\end{align*}
Now recalling the definitions of $b_{ij}$ in \eqref{def:b_ij} and of $\tau_j^+$ in \eqref{def:tau_j_plus} we get 
$$
\Delta = p \bigg[\sum\limits_{j = 1}^3 \theta_j \Big(\tfrac{\sigma_x^j}{|\sigma_x^j|^2\alpha_j}\Big)\bigg]^2 + p \bigg[\sum\limits_{j = 1}^3 
\nabla \theta_j \Big(\tfrac{\sigma_x^j}{|\sigma_x^j|}\Big) \cdot \tfrac{[\sigma_x^j]^\perp}{|\sigma_x^j|^2\alpha_j}\bigg]^2 \ne 0.
$$
Thus, $w = 0.$

Now we check the complementary conditions for the initial datum. Let $\cC$ be the $6\times6$-identity matrix. Note that at $t=0$ we have $\cC u = \sigma.$ We need to check that the rows of the matrix 
$$
\cD(x,p) = \cC\cdot \hat \cL(x,p,0)
=
\begin{pmatrix}
p^2 \bI & \bO & \bO \\
\bO & p^2 \bI & \bO \\
 \bO & \bO & p^2 \bI 
\end{pmatrix} 
$$
are linearly independent modulo $L(x,p,0) = p^6,$ which is obvious.

In view of \eqref{init_condition}  the linear problem satisfies the compatibility condition of order $0,$ and therefore, by the theory of linear parabolic systems \cite{Sol:1965}, there exists a unique solution $u\in \big[C_T^{\frac{2+\alpha}{2},2+\alpha} \big]^3$ of  \eqref{eq:parPDE}-\eqref{bound_cond_lin} satisfying 
\begin{equation}\label{resolv_estimate12}
\sum\limits_{j=1}^3 \|u^j\|_{C_T^{\frac{2+\alpha}{2},2+\alpha}} \le 
C_0\bigg[
 \sum\limits_{j=1}^3 \Big(\|F_{\bar u}^j\|_{C_T^{\frac{\alpha}{2},\alpha}} 
+ 
\|\sigma^j\|_{C^{2+\alpha}} \Big)
+ 
\|B_{\bar u}\|_{C_T^{\frac{1+\alpha}{2}}}
\bigg], 
\end{equation} 
where $C_0>0$ does not depend on $T.$

\subsubsection{Self-map property}

Now taking a larger $M>1$ and a smaller $T$ if necessary, we show that for any $\bar u\in \cR_{M,T}$ the unique solution $u$ of \eqref{eq:parPDE}-\eqref{bound_cond_lin} also belongs to $\cR_{M,T}.$

To this aim we estimate $\|F_{\bar u}^j\|_{C_T^{\frac{\alpha}{2},\alpha}}$ and $\|B_{\bar u}\|_{C_T^{\frac{1+\alpha}{2}}}$ in \eqref{resolv_estimate12}.
By the definition of $F_{\bar u}^j$ we have 
$$
\|F_{\bar u}^j\|_{C_T^{\frac{\alpha}{2},\alpha}} \le \|f^j\|_{C_T^{\frac{\alpha}{2},\alpha}}  + C_1\Bigg\| 
\Bigg(\frac{\beta_j(\frac{\bar u_x^j}{|\bar u_x^j|})}{|\bar u_x^j|^2} -
\frac{\beta_j(\frac{\sigma_x^j}{|\sigma_x^j|})}{|\sigma_x^j|^2}\Bigg) \Bigg\|_{C_T^{\frac{\alpha}{2},\alpha}} \|\bar u_{xx}^j\|_{C_T^{\frac{\alpha}{2},\alpha}}.
$$
Since $\beta_j\in C^{2+\alpha}$ out of the origin  and $|u|\ge \delta /2$ for all $u\in \cR_{M,T}$ (Lemma \ref{lem:regular_parametirze}),  by Lemma \ref{lem:holder_estimates_for_ai} 
$$
\Bigg\| 
\Bigg(\frac{\beta_j(\frac{\bar u_x^j}{|\bar u_x^j|})}{|\bar u_x^j|^2} -
\frac{\beta_j(\frac{\sigma_x^j}{|\sigma_x^j|})}{|\sigma_x^j|^2}\Bigg) \Bigg\|_{C_T^{\frac{\alpha}{2},\alpha}} \le 3C_2M^2\|\bar u_x^j - \sigma_x^j\|_{C_T^{\frac{\alpha}{2},\alpha}}
$$
for some $C_2>0$ depending only on $\delta,$ $\|\beta^j\|_\infty,$ $\|\nabla \beta^j\|_\infty$ and $\|\nabla^2\beta^j\|_\infty.$ 
Since $\bar u^j(0,\cdot) = \sigma^j,$ by the fundamental theorem of calculus and the choice of $\cR_{M,T}$ we have
$$
\|\bar u_x^j - \sigma_x^j\|_{C_T^{\frac{\alpha}{2},\alpha}} \le MT
$$
and therefore, taking into account also $\|\bar u_{xx}^j\|_{C_T^{\frac{\alpha}{2},\alpha}} \le M$ we get 
$$
\|F_{\bar u}^j\|_{C_T^{\frac{\alpha}{2},\alpha}} \le \|f^j\|_{C_T^{\frac{\alpha}{2},\alpha}}  + 3C_1C_2M^4T.
$$
Similarly, 
$$
\|B_{\bar u}\|_{C_T^{\frac{1+\alpha}{2}}} \le C_3M^4T
$$
for some constant $C_3$ depending only on $\|\theta_j\|_\infty$ $\|\nabla\theta_j\|_\infty,$ $\|\nabla^2\theta_j\|_\infty,$ $\|\nabla^3\theta_j\|_\infty$ and $\delta.$

Inserting these estimates in \eqref{resolv_estimate12} we get 
$$
\|u\|_{C_T^{\frac{2+\alpha}{2},2+\alpha}} \le C_0\sum\limits_{i=1}^3 \Big[ \|f^j\|_{C_T^{\frac{\alpha}{2},\alpha}} + \|\sigma^j\|_{C^{2+\alpha}} \Big]
+ C_0(3C_1C_2 + C_3)M^4T.
$$
Hence, if we choose 
$$
M :=   1 + 2 C_0\sum\limits_{i=1}^3  \Big[ \|f^j\|_{C_T^{\frac{\alpha}{2},\alpha}} + \|\sigma^j\|_{C^{2+\alpha}} \Big],
$$
then 
$$
\|u\|_{C_T^{\frac{2+\alpha}{2},2+\alpha}} \le M
$$
provided 
$$
T \le \frac{M + 1}{2C_0(3C_1C_2 + C_3)M^4}.
$$

\subsubsection{Contraction property}

Given $\bar u,\bar v\in\cR_{M,T},$ let $u=\cS_{\bar u}$ and $ v = \cS_{\bar v} \in \cR_{M,T}$ be the corresponding solutions to \eqref{eq:parPDE}-\eqref{bound_cond_lin}. Choosing $T$ smaller if necessary let us show that 
$$
\|\cS_{\bar u} - \cS_{\bar v}\|_{C_T^{\frac{2+\alpha}{2},2+\alpha}} \le 
\frac12\,\|\bar u - \bar v\|_{C_T^{\frac{2+\alpha}{2},2+\alpha}}.
$$
Let $w=u-v.$ 
Then $w$ solves the linear system 
$$
\begin{cases}
w_t^j - \alpha_j w_{xx}^j = F_{\bar u}^j - F_{\bar v}^j,\\
w^j(0,\cdot) = 0,
\end{cases}
$$
coupled with the boundary conditions
\begin{align*}
& w^1(t,0) =  w^2(t,0) = w^3(t,0), \\
&\sum\limits_{j=1}^3 \Bigg(\theta_j\Big(\tfrac{\sigma_x^j(y)}{|\sigma_x^j(y)|}\Big)\,\tfrac{w_x^j (t,y)}{|\sigma_x^j(y)|} + \bigg[\nabla\theta_j\Big(\tfrac{\sigma_x^j(y)}{|\sigma_x^j(y)|}\Big)\cdot \tfrac{[\sigma_x^j(y)]^\perp}{|\sigma_x^j(y)|}\bigg]\tfrac{[w_x^j(t,y)]^\perp}{|\sigma_x^j(y)|}\Bigg) \\
&\hspace*{1cm}= B_{\bar u}(t,y) - B_{\bar v}(t,y),
\end{align*}
for $y=0,1$. As we checked above, this linear system satisfies the compatibility and complementary conditions, and thus it  admits a unique solution, satisfying 
$$
\|w\|_{C_T^{\frac{2+\alpha}{2},2+\alpha}} \le 
C_0\bigg[\Big(\|F_{\bar u} - F_{\bar v} \|_{C_T^{\frac{\alpha}{2},\alpha}} 
+  
\|B_{\bar u} - B_{\bar v} \|_{C_T^{\frac{1+\alpha}{2}}}
\bigg].
$$
Therefore, repeating the same arguments above we find 
$$
\|w\|_{C_T^{\frac{2+\alpha}{2},2+\alpha}} \le C_4 T\|\bar u - \bar v\|_{C_T^{\frac{2+\alpha}{2},2+\alpha}},
$$
where $C_4$ depends only on $\delta,$ $M,$ $\beta_j$ and $\theta_j.$ Now possibly reducing $T>0$ if necessary we deduce the required contraction property.

\subsubsection{Proof of Theorem \ref{teo:existence_special_mcf}}

Finally, using the Banach fixed point theorem we conclude that there exists a unique $u\in\cR_{M,T}$ solving   system \eqref{existed_flow}. 

\subsection{Evolution of networks in the Euclidean setting}

Starting from the work  \cite{BR:1993}, a vast literature is dedicated to the curvature-driven flow of networks
(see e.g. \cite{INS:2014,MMN:2016,MNP:2017,MNPS:2016,MNT:2004} and references therein). In this section we shortly describe known results related to evolution of networks in the Euclidean setting; we refer to the recent survey \cite{MNPS:2016} for more details.

In the Euclidean setting, the condition
\eqref{herring_condosh_normals} at the triple junctions reduces to a $120^\circ$-condition between normals. The existence and regularity of a flow with Dirichlet boundary conditions has been established, for instance, in \cite{BR:1993,MNT:2004}. Here one needs to assume the $120^\circ$-condition at all triple junctions of the initial network.

The behaviour of such a smooth flow near the maximal time, as in the  curvature evolution of closed curves, is obtained using integral estimates for the curvature. Namely,

\begin{theorem}
Let $\{\Sigma_t\}_{t\in[0,T)}$ be the smooth geometric flow in the maximal time interval $[0,T),$ starting from an (admissible) network $\Sigma_0$ in a bounded convex open set $\Omega$ with Dirichlet boundary conditions on $\p\Omega$. Then: 
\begin{itemize}
 \item either the lower limit of the length of at least one curve in $\Sigma_t$ converges to $0$ as $t\nearrow T,$
 
 \item or 
 $$
 \limsup\limits_{t\nearrow T} \int_{\Sigma_t} \kappa^2\,d x = +\infty.
 $$
\end{itemize}
Moreover, if the lengths of all curves in $\Sigma_t$ are uniformly bounded away from zero as $t\nearrow T,$ then there exists $C>0$ such that 
\begin{equation}
 \int_{\Sigma_t} \kappa^2\,d x \ge \frac{C}{\sqrt{T-t}}\qquad\text{for all $t<T.$}
 \label{vrachlar_oyligi}
\end{equation}

\end{theorem}

Note that under the uniform lower bound on the length,  \eqref{vrachlar_oyligi} implies 
\begin{equation}\label{saldoadad}
\max\limits_{\Sigma_t}\,\kappa^2 \,\ge \, \frac{C'}{\sqrt{T-t}}\qquad\text{for all $t<T,$} 
\end{equation}
where $C'$ depends also on the lengths of the curves, which is slightly weaker than in the blow up case of closed curves $\{\gamma_t\}_{t\in[0,T)}$, which reads as
$$
\max\limits_{\gamma_t}\,\kappa^2 \,\ge \, \frac{C}{T-t}\qquad\text{for all $t<T.$}
$$
However, the latter estimate for networks is not known even for simple triods (three smooth curves with a single triple junction).  

The next question is the blow-up behaviour of the rescaled networks near the maximal time. 
As in the closed curves setting, one can establish Huisken's monotonicity formula for networks \cite{INS:2014,MNPS:2016} and then, using parabolic rescaling, one approaches some limiting network as $t\nearrow T$. A complete  classification of these limiting networks is a hard problem, because they could  be not regular; for example, we may loose the $120^\circ$ condition (collapse of two triple junctions), two curves of the network may collapse (higher-multiplicity), or even some phase  may collapse to a point or segment. Of course, one of the possibilities are self-shrinking networks; for their classification we refer to \cite{BHM:2017,BHM:2017_2,ChG:2007,MNPS:2016}.

\subsection{Long time behaviour of the anisotropic flow}

We recall from \cite{KNP:2021} that, as in the Euclidean curvature flow of networks, at the maximal time in a triod with Dirichlet boundary conditions either some of the curves disappear or the  curvature of some curve blows up and also near the maximal time the $L^2$-norm  of the curvature satisfies \eqref{saldoadad} 
provided that the length of the curves is uniformly bounded away from zero.

However, to our knowledge, not much is known on theta-shaped networks near the maximal existence time even in the case of a single noneuclidean anisotropy: in this case we cannot straightforwardly repeat/adapt the arguments of \cite{KNP:2021} because (as in the isotropic case) we could  have not only singularities related to the blow-up of the curvature or disappearance of a curve, but also collapse of triple points or region disappearance.
Moreover, the problem of existence of homothetically shrinking theta-shaped networks seems open in the anisotropic case. Recall that in the isotropic case such a homothetic network does not exist  \cite{BHM:2018}.

\section{Crystalline curvature flow of networks}

The classical definition of curvature in the smooth case breaks down if we lack the smoothness of the anisotropy, for instance in with crystalline case. As in the two-phase case \cite{Bellettini:2004,BCChN:2005,BNP:1999,GP:2016,GP:2018,GP:2022,Taylor:1978,Taylor:1991,Taylor:1992,Taylor:1993}, the crystalline curvature becomes nonlocal and its definition requires a special class of networks, admitting a Cahn-Hoffman vector field.

In this section we extend the definition of smooth anisotropic curvature flow to the polycrystalline
case, generalizing  \cite{BCN:2006}. Unless otherwise stated, in what follows we only consider even crystalline anisotropies. 

For simplicity, we assume that any curve we consider is polygonal, consisting of finitely many (at least one) segments and at most one half-line, having fixed its unit normal (via parametrization). 

\begin{definition}\label{def:some_notions}
$\,$
\begin{itemize}
\item[(a)] {\bf Distance vector between two parallel lines and segments/half-lines.} Let $L_1$ and $L_2$ be two parallel lines. A vector $H$ is called a \emph{distance vector of $L_2$ from $L_1$} if $|H|  = \dist(L_1,L_2)$ and $x_0 + H\in L_2$ for any $x_0\in L_1$. In other words, $L_2 = L_1 + H.$ 
Similarly, given two parallel segments/half-lines $S$ and $T$ a vector $H$ is a  \emph{distance vector of $S$ from $T$} if $H$ is the distance vector of the line containing $T$ from that of $S.$ We write 
$$
H(S,T)
$$
to denote the distance vector of a segment/half-line $S$ from a segment/half-line $T.$

\item[(b)] {\bf Parallel networks.} Let $\Sigma=\bigcup_{i=1}^n\Sigma_i$ be a polygonal network such that each curve $\Sigma_i$ consists of $m_i\ge0$ segments $S_1^i,\ldots, S_{m_i}^i$ (in the increasing order of parametrization\footnote{i.e. each $S_j^i$ starts from the point where $S_{j-1}^i$ ends.} of $\Sigma_i$) and $l_i\in\{0,1\}$ half-lines $L_i$, $m_i+l_i\ge1$. We say that a network $\bar \Sigma$ is a \emph{parallel to $\Sigma$} provided that:
\begin{itemize}
\item it consists of $n$ embedded curves $\bar \Sigma_1,\ldots,\bar \Sigma_n;$

\item for each $i$ the curve $\bar \Sigma_i$ consists of $m_i$ segments $\bar S_1^i,\ldots, \bar S_{m_i}^i$ (in the same order as in $\bar\Sigma_i$) and $l_i$ half-lines  $\bar L_i;$

\item for each $i\in\{1,\ldots,n\}$ the segments $S_j^i$ and $\bar S_j^i$  are parallel for all $j$  and $\bar L_i$ and $L_i$ lie on the same line;

\item if $q$ is a junction of $\Sigma_{i_1},\ldots, \Sigma_{i_k}$ for some $k\ge3,$ then $\bar \Sigma_{i_1},\ldots, \bar\Sigma_{i_k}$ form a junction in the same order as $\{\Sigma_{i_j}\}.$
\end{itemize}

\item[(c)] {\bf Distance between parallel networks.} Let $\Sigma$ and $\bar \Sigma$ be parallel networks. We set
$$
d(\Sigma,\bar\Sigma): = \max\limits_{i,j} |H(S_j^i, \bar S_j^i)|. 
$$
\end{itemize}

\end{definition}

\begin{figure}[htp!]
\begin{center}
\includegraphics[width=0.8\textwidth]{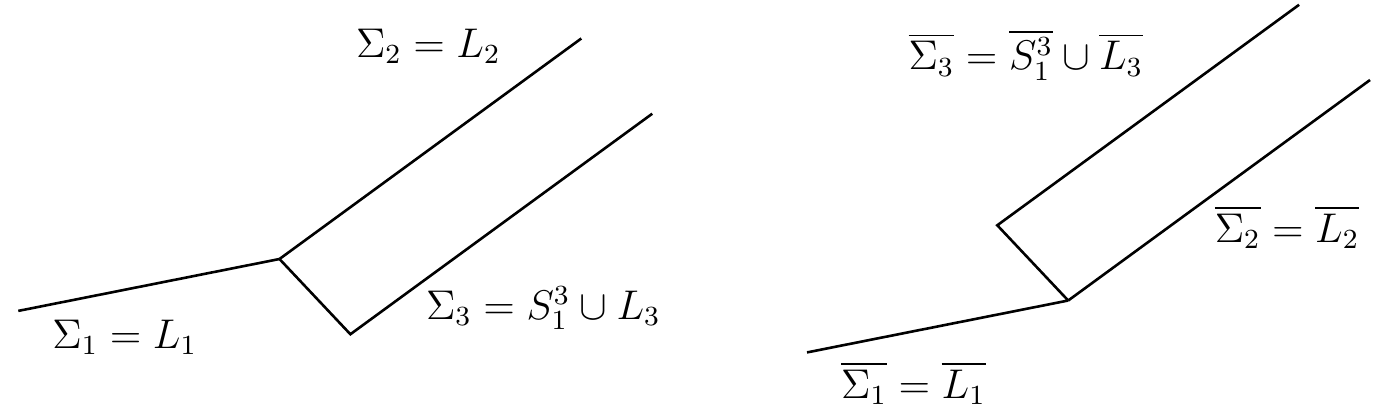}
\caption{\small Non-parallel networks satisfying the first three assumptions of Definition \ref{def:some_notions}(b).}
\label{fig:nonparallel_netwrok}
\end{center}
\end{figure}

\begin{remark}\label{rem:notions_property}
$\,$ 
\begin{itemize}
\item[(a)] If $\nu_S=\nu_T,$ then $H(S,T) = -H(T,S).$ 

\item[(b)] Two parallel networks have the same structure and only the length of segments and/or endpoints of half-lines may differ. The condition on junctions in Definition \ref{def:some_notions} prevents the situations drawn in Figure \ref{fig:nonparallel_netwrok}. 

\item[(c)] Given a network $\Sigma$  and a sequence $\{\Sigma(k)\}$ of networks parallel to $\Sigma$ the following assertions are equivalent:

\begin{itemize}
 \item[(1)] $d(\Sigma,\Sigma(k)) \to 0;$
 
 \item[(2)] each segment $S_j^i(k)$ of $\Sigma(k)$ converges to the corresponding segment $S_j^i$ of $\Sigma$ in the Kuratowski sense;
 
 \item[(3)] $\cH^1(S_j^i(k)) \to \cH^1(S_j^i)$ for all $i$ and $j.$
\end{itemize}

\end{itemize}

\end{remark}

Given parallel networks $\Sigma$  and $\bar\Sigma,$ the distance vectors of the segments of $\bar \Sigma$ from those of $\Sigma$ are uniquely defined.

\begin{proposition}\label{prop:height_vs_length}

Let $\Phi=\{\phi_i\}_{i=1}^n$ be crystalline anisotropies, $\Sigma$ and $\bar\Sigma$ be two parallel polygonal networks and let $J$ and $\bar J$ be two corresponding segments of $\Sigma$  and $\bar\Sigma.$ Let us write 
$$
h_T: = H(\bar T, T) \cdot \nu_T,
$$ 
where $T$  (resp. $\bar T$) is a segment/half-line in $\Sigma$ (resp. $\bar \Sigma$).
\begin{itemize}
\item[\rm(a)] Let $J$ do not end at a triple junction and let $S',S''$ be segments/half-lines of $\Sigma$ which end at the  endpoints of $J.$ Then 
$$
\cH^1(\bar J) = \cH^1(J) + a |h_{J}| + b |h_{S'}| + c |h_{S''}|,
$$
where $a,b,c$ are real numbers depending only on the angles between $S',S$ and $S,S''.$  

\item[\rm(b)] Let $J=[AB]$ and $B$ be a triple junction of $J$ and two other segments/half-lines $S'$ and $S''.$

\begin{itemize}
\item[\rm(b1)] Let a segment/half-line $T$ of $\Sigma$ end at $A.$ Then 
\begin{equation}\label{one_tripod_length}
\cH^1(\bar J) = \cH^1(J) + ah_T + b h_J + c h_{S'} + dh_{S''}, 
\end{equation} 
where $a,b,c,d$ are real numbers, $a$ depends only on the angle between $\nu_T$ and $\nu_S,$ and $b,c,d$ depend only on the angles between $J,$ $S'$, $S''.$

\item[\rm(b2)] Let $A$ be another triple junction of $J$ and segments/half-lines $T'$ and $T''.$  Then 
\begin{equation}\label{two_tripod_length}
\cH^1(\bar J) = \cH^1(J) + ah_{T'} + b h_{T''} + ch_J + dh_{S'} + eh_{S''}, 
\end{equation} 
\end{itemize}
where $a,b,c,d,e$ are constants depending only on the angles between $T',T'',J,$ and $J,S',S''.$ 

\end{itemize}
 
\end{proposition}

\begin{proof}

(a) In view of the signs of $h_{S'},$ $h_J$ and $h_{S''}$ we have eight possible configurations for the relative location of $\bar J$ and $J$ (Figure \ref{fig:eight_possibility}).

\begin{figure}[htp!]
\begin{center}
\includegraphics[width=0.8\textwidth]{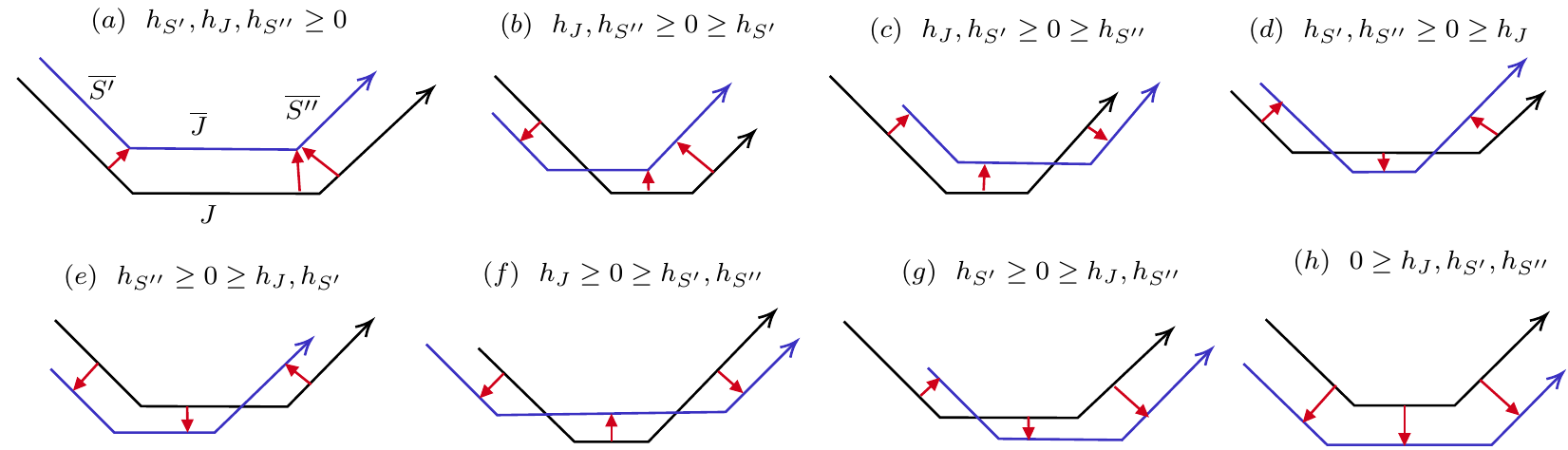}
\caption{\small Eight possible (schematic) configurations of $\bar J$ and $J.$}
\label{fig:eight_possibility}
\end{center}
\end{figure}

For simplicity, let us compute the length of $\bar J$ in case (g) of Figure \ref{fig:eight_possibility}, i.e., $h_{S'}\ge 0\ge h_J,h_{S''}$ see Figure \ref{fig:comp_len}. Clearly, 
$$
\cH^1(\bar J) = \cH^1(J) -\cH^1([AE])| + \cH^1([G\bar B]).
$$

\begin{figure}[htp!]
\begin{center}
\includegraphics[width=0.8\textwidth]{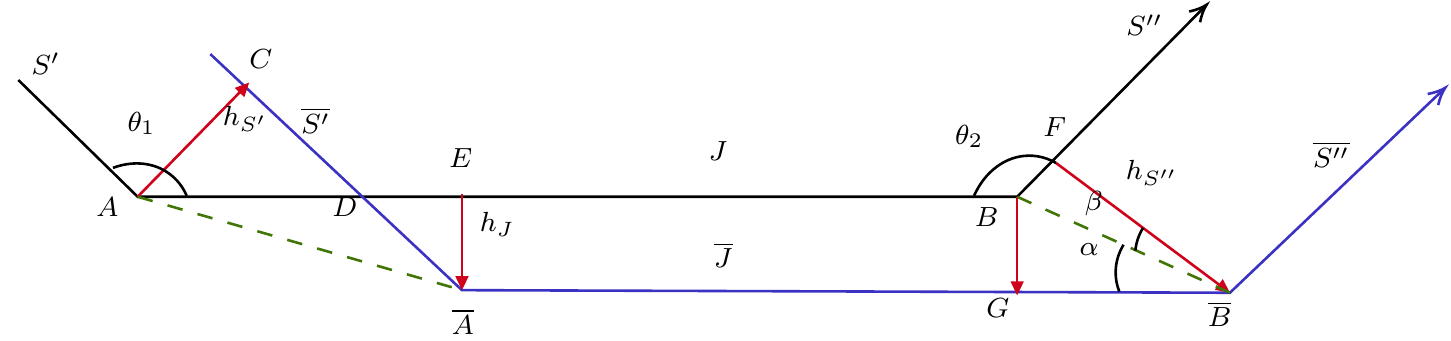}
\caption{\small Computing the length of segment $\bar J$ not ending at a triple junction.}
\label{fig:comp_len}
\end{center}
\end{figure}
Let us denote the angles of $\Sigma$ at $A$ and $B$  by $\theta_1$ and $\theta_2;$ obviously, these two angles are uniquely determined by $\nu_{S'},$ $\nu_J$ and $\nu_J,\nu_{S''},$ respectively. As both the angles $\angle CAD$ and $\angle D\bar AE$ are equal to $\theta_1 - \pi/2,$ 
\begin{multline*}
\cH^1([AE]) = \cH^1([AD]) + \cH^1([DE]) \\ 
= \frac{h_{S'}}{\cos(\theta_1 - \pi/2)} - h_J\tan(\theta_1 - \pi/2) =\frac{h_{S'}}{\sin\theta_1} + h_J\cot\theta_1.
\end{multline*}
Since 
$$
\cH^1([B\bar B]) = \frac{\cH^1([F\bar B])}{\cos \angle F\bar B B} = 
\frac{\cH^1([GB])}{\sin \angle GB\bar B} \quad \text{and}\quad \angle G\bar B F = \theta_2 - \pi/2 = \alpha + \beta,
$$
we find 
$$
-\frac{h_J}{\sin\alpha} = -\frac{h_{S''}}{\cos(\theta_2 - \pi/2 - \alpha)}
$$
or equivalently,
$$
h_J(\sin\theta_2\cot\alpha - \cos\theta_2) = h_{S''}.
$$
This implies
$$
\cH^1([G\bar B]) = - h_J\cot \alpha = -\frac{h_{S''}}{\sin\theta_2} - h_J\cot\theta_2.
$$
Thus, 
$$
\cH^1(\bar J) = \cH^1(J) - \frac{h_{S'}}{\sin\theta_1} + h_J(\cot\theta_1 -\cot\theta_2) - \frac{h_{S''}}{\sin\theta_2}.
$$

(b1) As in case (a) we can consider all possible relative configurations of $(S',J,S'')$ and $(\bar S',\bar J, \bar S'')$  (more than 27 cases). For simplicity, let us assume that we are as in Figure \ref{fig:ending_tripod} and let
$$
x_1:=|h_J| = \cH^1([BD]),\quad
x_2:=|h_{S'}| = \cH^1([\bar BF]),\quad
x_3:=|h_{S''}| = \cH^1([\bar BE]).
$$

\begin{figure}[htp!]
\begin{center}
\includegraphics[width=0.7\textwidth]{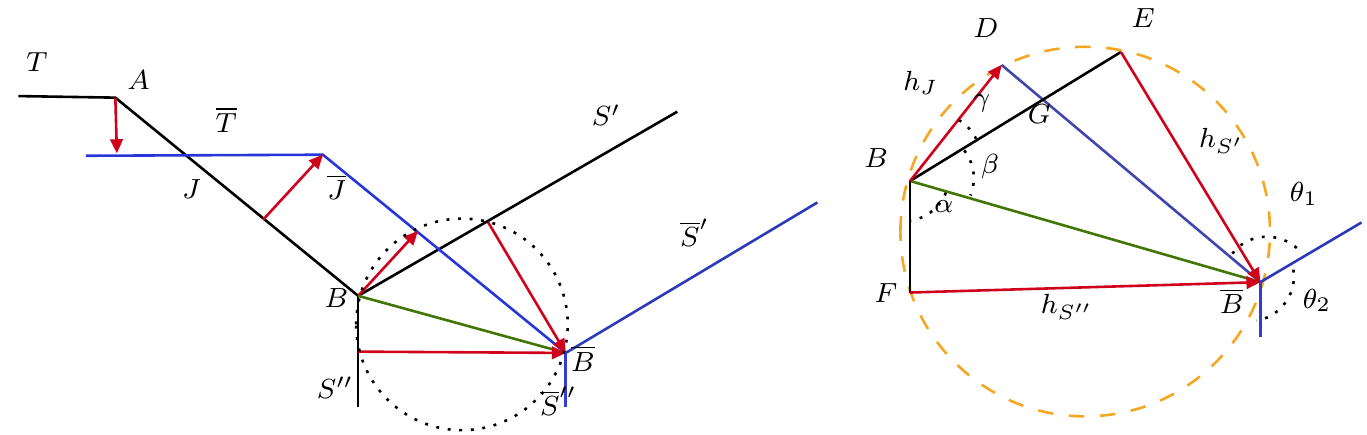}
\caption{\small Computing the length of the segment ending at a triple junction $B$ (case (b1) of Proposition \ref{prop:height_vs_length}).}
\label{fig:ending_tripod}
\end{center}
\end{figure}

We just need to compute $\cH^1([\bar BD]).$
Since $\gamma = \theta_1-\pi/2,$ we have
$$
\cH^1([\bar BD]) = \cH^1([DG]) + \cH^1([G\bar B]) = x_1 \tan\gamma + \frac{x_3}{\cos\gamma} = \frac{x_3}{\sin\theta_1} - x_1\cot\theta_1.
$$
Analogous computations can be done near $A$ and  $\bar A.$ Now observing that $x_1 = h_J,$ $x_2 = -h_{S'}$ and $x_3 = -h_{S''},$ we deduce \eqref{one_tripod_length}. Applying (b1) at each triple point we conclude \eqref{two_tripod_length}.
\end{proof}

\subsection{Cahn-Hoffman vector fields associated to a Lipschitz curve}

Let $\Sigma$ be an embedded Lipschitz curve and let $\phi$ be an anisotropy. We denote by $\Lip_\phi(\Sigma; \R^2)$ the set of all vector fields $N\in \Lip(\Sigma; \R^2)$ such that 
\begin{equation}\label{can_hopmanchuk}
\phi(N)=1\qquad\text{and}\qquad N\cdot \nu_\Sigma = \phi^o(\nu_\Sigma)\quad \text{$\cH^1$-a.e. on $\Sigma.$} 
\end{equation}
Any such vector $N$ is called a \emph{Cahn-Hoffman vector field}. Note that \eqref{can_hopmanchuk} is equivalent to saying $N\in \p \phi^o(\nu_\Sigma)$ $\cH^1$-a.e. on $\Sigma,$ where $\p \phi^o$ is the subdifferential of $\phi^o$. We recall that not every Lipschitz curve admits a Cahn-Hoffman vector  field. However, when it exists, we call the curve $\phi$-regular.

\subsection{$\Phi$-regular networks}\label{subsec:cahn_hoffman_definiton}

Let $\Sigma\subset \R^2$ be a network consisting of polygonal Lipschitz curves\footnote{As in the smooth case, we write $\Sigma_k$ and $\phi_k$ in places of $\Sigma_{ij}$ and $\phi_{ij}.$}  $\{\Sigma_i\}_{i=1}^n$ and let $\Phi = \{\phi_i\}_{i=1}^n$ be a set of anisotropies with $\phi_i$ associated to $\Sigma_i$. 
We denote by $\Lip_\Phi(\Sigma; \R^2)$ the space of all vector fields $N:\Sigma \to  \R^2$ such that $\restriction N|{\Sigma_i}\in \Lip_{\phi_i}(\Sigma_i;  \R^2).$ We set 
\begin{equation}\label{admissible_cahnhopman}
\cN^\Sigma: = \Big\{N\in\Lip_\Phi(\Sigma;  \R^2):\quad\sum_{j=1}^m (\restriction N|{\Sigma_{i_j}}(q))^{\p\Sigma_{i_j}} = 0\quad\text{at the $m$-tuple point $q$} \Big\}, 
\end{equation}
where $\restriction N|{\Sigma_i}^{\p\Sigma_i}$ is defined in \eqref{conormal_tonagentos} and $q$ is endpoint of (exactly) $m$ curves $\Sigma_{i_j}.$ Any element of $\cN^\Sigma$ is called a \emph{Cahn-Hoffman vector field associated to $\Sigma$} and any network admitting at least one Cahn-Hoffman vector field is called a \emph{$\Phi$-regular network}. 

The condition on junctions in \eqref{admissible_cahnhopman} is called \emph{balance condition}\footnote{Which is a version of Herring condition.}.

In what follows we are mainly concerned with networks with triple junctions, and hence, in the balance condition only three vectors appear. Given three anisotropies $\phi_1,\phi_2,\phi_3$, let us call any triplet $(X_1,X_2,X_3)$ such that $X_i\in\p B_{\phi_i}$ and 
$$
X_1 + X_2 + X_3 = 0
$$
\emph{admissible}. We anticipate here that, unlike the elliptic case, the admissible triplets at a triple junction coud be even uncountably many (see Lemma \ref{lem:tripletchuk} below).

\begin{figure}[htp!]
\begin{center}
\includegraphics[width=0.5\textwidth]{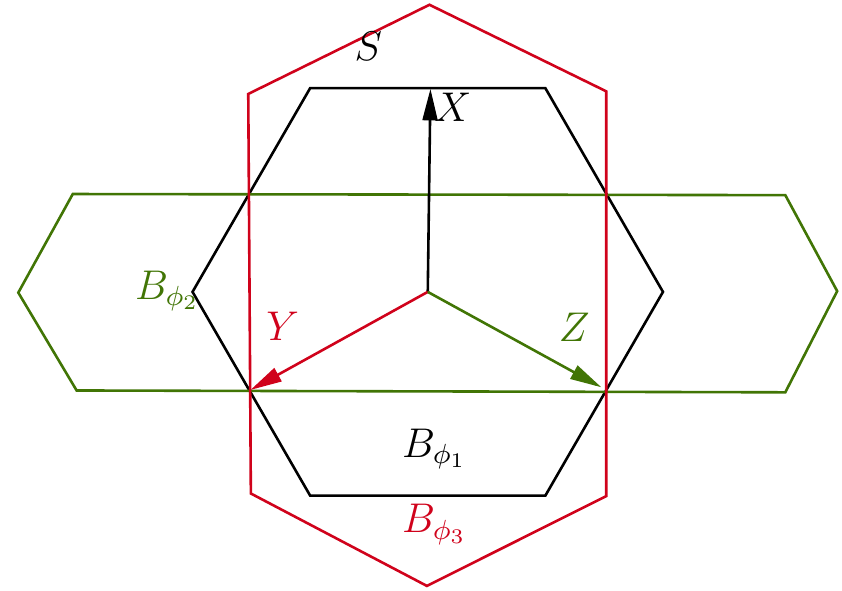}
\caption{\small An admissible triplet for three hexagonal anisotropies $\phi_1,\phi_2,\phi_3.$ The boundaries of hexagons $B_{\phi_2}$ and $B_{\phi_3}$ cross the boundary segments of $B_{\phi_1}$ at their midpoints. Therefore, if $X\in S$ is not the  midpoint, then there are no $Y\in\p B_{\phi_2}$ and $Z\in\p B_{\phi_3}$ such that $X + Y + Z = 0.$}
\label{fig:three_bad_hex}
\end{center}
\end{figure}

In the case of different anisotropies, showing $\cN^\Sigma\ne \emptyset$ is not trivial (see Figure \ref{fig:three_bad_hex}).

\begin{figure}[htp!]
\begin{center}
\includegraphics[width = 0.8\textwidth]{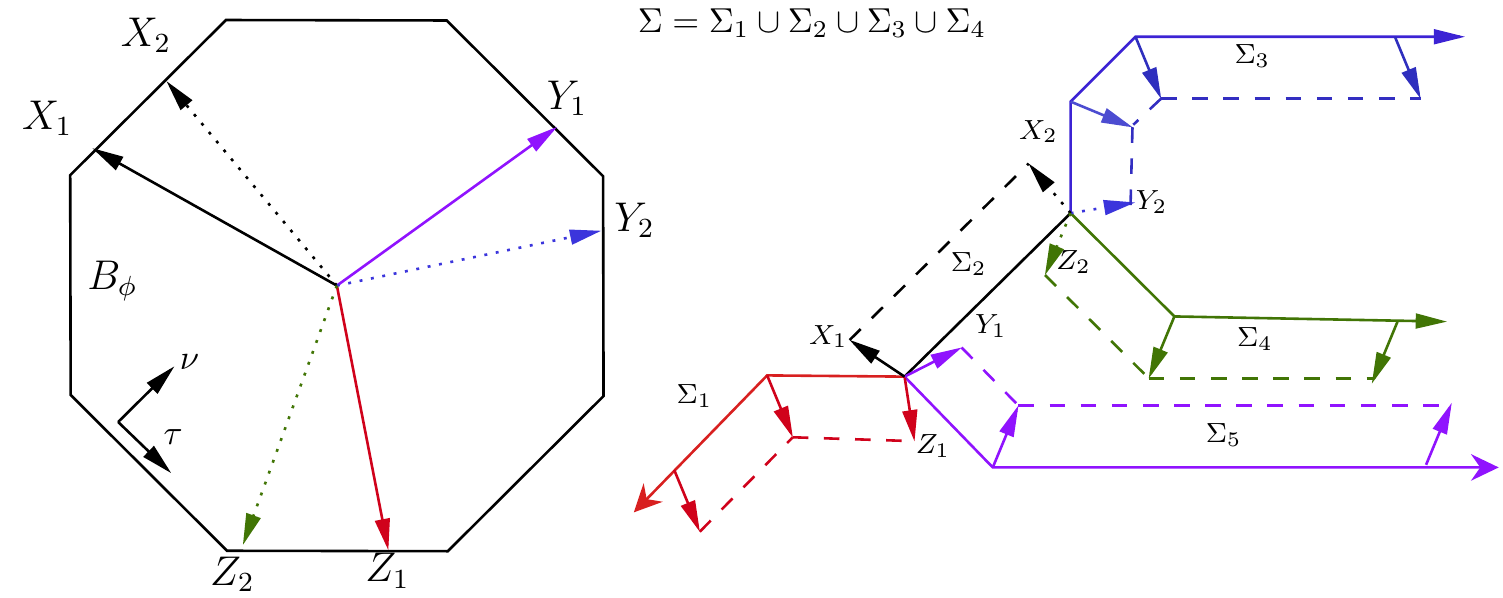} 
\caption{\small A possible definition of Cahn-Hoffman vector field using only its values at the vertices and at triple junctions in the case of a single anisotropy. Note that the Cahn-Hoffman vector is uniquely defined at the vertices of $\Sigma_i$.}
\label{fig:ext_cahn_hoff}
\end{center}
\end{figure}

\begin{remark}\label{rem:cahn_hoffman_field09}
Let $\Phi=(\phi_1,\ldots,\phi_n)$ be crystalline anisotropies and let $\Sigma=\bigcup_{i=1}^n\Sigma_i$ be a $\Phi$-regular polygonal network.

\begin{itemize}
\item If $S=[AB]$ is a segment of $\Sigma_i$ and $N\in\cN^\Sigma,$  then $N(A) - N(B)$ is parallel to the tangent vector $\tau_S$ to $S.$ Moreover, $N(A)$ and $N(B)$ belong  to the same edge of the Wulff shape $B_{\phi_i},$ whose tangent is parallel to $\tau_S.$ In particular, if $S$ does not end at an $m\ge3$-tuple junction, then any Cahn-Hoffman vector field $N$ is uniquely defined at the endpoints $A$ and $B$ of $S$ and hence, it can be extended along $S$ in a Lipschitz way keeping \eqref{can_hopmanchuk} valid.

\item If $\Sigma$ contains a half-line $L$ with endpoint at $A$, then $N$ can be defined along $L$ constantly equal to $N(A).$ Similarly, if $\Sigma$ contains a  ``curved'' part\footnote{A part which is not parallel to some of the sides of the corresponding Wulff shape.} $C$ \cite{BCN:2006}, then $N$ can be taken constant along $C.$

\end{itemize}
\end{remark}

\subsection{Crystalline curvature of a $\Phi$-regular network}

The following result is an  improvement of \cite[Theorem 4.8]{BNR:2003}  and can be shown along the same lines.

\begin{theorem} 
Let $\Phi=(\phi_1,\ldots,\phi_n)$ be crystalline anisotropies associated to a network $\Sigma = (\Sigma_1,\ldots,\Sigma_n).$ If $\Sigma$ is $\Phi$-regular, then the minimum problem 
\begin{equation}\label{eNmin}
\min \left\{ \sum\limits_{i=1}^n \int_{\Sigma_i} 
 \big[\div_\Sigma N\big]^2\, \phi_i^o(\nu)~d\cH^1:\,  N\in \cN^\Sigma \right\}
\end{equation}
admits a unique  solution\footnote{which identifies the direction
along which the length functional  \eqref{network_length} decreases ``most quickly''.} $N_{\min}$.
\end{theorem}

Recall that in \cite[Theorem 4.8]{BNR:2003} the authors assume that all $\phi_i$ are equal.

\begin{definition}
Let $\Sigma$ be a $\Phi$-regular network. We define the $\Phi$-curvature $\kappa^\Phi$ of  $\Sigma$ as 
$$ \kappa^{\Phi}:= \div_\Sigma N_{\min}\quad\text{a.e.\, on\, $\Sigma$.}
$$
Sometimes we denote the $\Phi$-curvature by $\kappa_\Sigma^\Phi$ if we want to emphasize its dependence on $\Sigma.$
\end{definition}

We recall that in the two-phase case the structure of $N_{\min}$ over a planar Lipschitz $\phi$-regular curve $\Sigma$ giving the curvature is known. Namely, if the curve is polygonal, then the values of $N_{\min}$ are uniquely defined as the linear interpolation of its values at the vertices. Moreover, if $\Sigma$ is not polygonal, then $N_{\min}$ is constant on curved parts. 

\begin{remark}\label{rem:prop_curvature812}
Unlike the smooth case, the crystalline curvature of a network is nonlocal. Still:

\begin{itemize}
\item[(a)] If a network $\Sigma$ contains a segment $S=[AB]$ not ending at an $m\ge3$-tuple junction, then $N_{\min}$ is uniquely defined at $A$ and $B$ and linear along $S,$ and hence
\begin{equation}\label{curvature_segments}
\kappa^\Phi = \frac{N_{\min}(B) - N_{\min}(A)}{\cH^1(S)}\cdot \tau_S \quad\text{on}\quad S,
\end{equation}
where $\tau_S$ is the tangent to the segment $S.$

\item[(b)] If $\Sigma$ contains a half-line $L$, then $\kappa^\Phi = 0$ on $L.$ Similarly, if $\Sigma$ contains a ``curved'' part $C,$ then $N_{\min}$ must be constant along $C$ and hence, $\kappa^\Phi = 0$ on $C.$ 

\item[(c)] As we mentioned in Remark \ref{rem:cahn_hoffman_field09}, a Cahn-Hoffman vector field $N$ can be defined only by its values at the triple junctions and at the vertices of segments of $\Sigma$ (see Figure \ref{fig:ext_cahn_hoff}). Then the minimizer of \eqref{eNmin} can be searched only among all possible values of $N$ at the triple junctions. 

\item[(d)] $\kappa^\Phi$ is constant on each segment $S$ and half-line $L$ of $\Sigma$, and we denote their curvatures by $\kappa^\Phi(S)$ and $\kappa^\Phi(L)=0,$ respectively.

\end{itemize}
\end{remark}

These observations imply the following properties of a $\Phi$-regular network. 

\begin{lemma}\label{lem:curvature_parallel_network}
Let $\Phi=(\phi_1,\ldots,\phi_n)$ be crystalline anisotropies and let $\Sigma :=\bigcup_{i=1}^n\Sigma_i$ be a $\Phi$-regular polygonal network. Then: 
\begin{itemize}
\item[\rm (a)] any network $\Sigma'$  parallel to $\Sigma$ is  $\Phi$-regular; 

\item[\rm(b)] let $q_1,\ldots,q_m$ be the multiple junctions of $\Sigma$, i.e., for each $q_l$ there exist  $n_l\ge3$ polygonal curves containing $q_l$ at their boundaries. Let $S_l^{i_1},\ldots,S_l^{i_{n_l'}}$ with $n_l'\le n_l$ be all segments ending at $q_l$ (hence $q_l$ belongs to $n_l-n_l'$ half-lines of $\Sigma$), where $S_l^i\subset \Sigma_i$. Then \eqref{eNmin} is equivalent to the minimum problem
\begin{equation}\label{dajdada}
\min\limits_{N\in\cN^\Sigma}\,\,\sum\limits_{l=1}^m \sum\limits_{j=1}^{n_l'} \phi_{i_j}^o(\nu_{S_l^{i_j}})\,\frac{[(N(B_l^{i_j}) - N(A_l^{i_j})) \cdot \tau_{S_l^{i_j}}]^2}{\cH^1(S_l^{i_j})}, 
\end{equation}
where $S_l^{i_j} = [A_l^{i_j}B_l^{i_j}];$

\item[\rm(c)] let $\{\Sigma(k)\}$ be a sequence of networks parallel to $\Sigma$ (so that by (a) each $\Sigma(k)$ is $\Phi$-regular) 
such that $d(\Sigma(k),\Sigma)\to0.$
Let $N_{\min}$ and $N_{\min}(k)$ be the solutions of \eqref{eNmin} applied with $\Sigma$ and $\Sigma(k)$. Then 
\begin{equation}\label{conve_cahnhoffman}
\restriction {N_{\min}(k)}|{S_i^j(k)}[\zeta(S_i^j(k),S_i^j,x)] \to \restriction {N_{\min}}|{S_i^j}[x]
\end{equation}
uniformly in $x\in S_i^j,$ and 
\begin{equation}\label{conve_curbature}
\kappa_{\Sigma(k)}^\Phi(S_i^j(k)) \to \kappa_{\Sigma}^\Phi(S_i^j)\quad\text{as $k\to+\infty,$}
\end{equation} 
where $S_j^i(k)$ and $S_j^i$ are corresponding parallel segments of $\Sigma(k)$ and $\Sigma,$ and $z(S,T,\cdot):T\to S$ is the linear bijection of $T$ onto $S,$ preserving the orientation.
\end{itemize}

\end{lemma}

\begin{proof}
(a) Let $N_{\min}^\Sigma\in \cN^\Sigma$ be the solution of \eqref{eNmin} and let us define $N^{\bar\Sigma}$ as follows: at vertices of $\bar \Sigma$ and also at multiple junctions we define $N^{\bar\Sigma} = N_{\min}^{\Sigma}$ and then we linearly interpolate them along segments/half-lines. Then such vector field belongs to $\cN^{\bar\Sigma}$.

(b) Since all segments not ending at multiple junctions admit  a unique minimizing Cahn-Hoffman field, the minimum problem is reduced to only segments ending at multiple junctions. Since $\kappa^\Phi$ is constant along all segments and $\kappa^\Phi =0$ on half-lines, the minimum problem \eqref{eNmin} is equivalent to \eqref{dajdada}.

(c) Assertion \eqref{conve_cahnhoffman} follows from the definition of Cahn-Hoffman at the vertices of a network, the definition of parallel networks and assertion (3) in Remark \ref{rem:notions_property}(c). Similarly, \ref{conve_curbature} follows from \eqref{conve_cahnhoffman}, \eqref{dajdada} and assertion (3) in Remark \ref{rem:notions_property}(c).
\end{proof}

\begin{remark} 
$\,$
\begin{itemize}
\item[(a)] The sum in \eqref{dajdada} is a function of $N(X)\cdot\tau_S,$ where $X$ is the endpoint of the segment $S$ at the junction. Since $N(X)$ varies in a compact subset of $\p B_\phi,$ the problem \eqref{eNmin} is the minimization of a quadratic function of finitely many variables $x:=N(X)\cdot \tau_S$ (depending on the number of junctions and their multiplicity) and subject to the balance condition in  \eqref{admissible_cahnhopman}. In particular, the unique minimizing field can be found depending only on anisotropies, the location  and length of segments ending at the junctions.

\item[(b)] If $\Sigma$ and $\bar\Sigma$ are parallel $\Phi$-regular networks  having only triple junctions, then Proposition \ref{prop:height_vs_length} and Lemma \ref{lem:curvature_parallel_network} allow to rewrite the minimum problem \eqref{dajdada} depending only on the distance vectors from the segments/half-lines of $\Sigma$.

\item[(c)] Recall that by our convention, no half-line of $\Sigma$ ends at a multiple junction.
\end{itemize}

\end{remark}

\subsection{Polycrystalline curvature flow}

In this section we define polycrystalline curvature flow of networks, generalizing \cite{BCN:2006}. For simplicity, we only consider networks without ``curved'' parts and with only triple junctions.

\begin{definition}[\textbf{Admissible network}] 
Given crystalline anisotropies $\Phi := (\phi_1,\ldots,\phi_n),$ let us call a network $\Sigma:=\bigcup_{i=1}^n \Sigma_i$ \emph{admissible} if 
\begin{itemize}
\item $\Sigma$ is $\Phi$-regular;

\item any multiple junction of  $\Sigma$ is a triple junction;

\item each $\Sigma_i$ consists of $m_i\ge1$ segments $S_1^i,\ldots,S_{m_i}^i$ (counted in the increasing order of parametrization of $\Sigma_i$) and at most one half-line $L_i;$ let $l_i\in\{0,1\}$ be the number of half-lines.
\end{itemize}
\end{definition}

\begin{figure}[htp!]
\begin{center}
\includegraphics[width=0.6\textwidth]{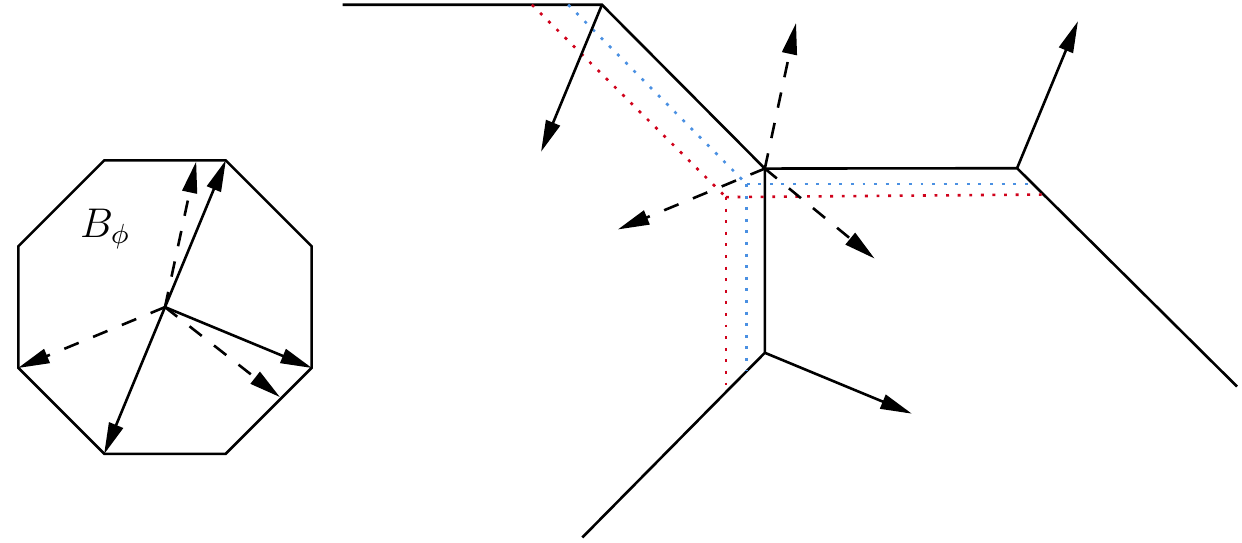} 
\caption{\small Evolution of a network with a single anisotropy.} \label{fig:evolution_network1}
\end{center}
\end{figure}

\begin{definition}[\textbf{Polycrystalline curvature flow of networks}] \label{def:curvature_flow}

Let $\Phi:=(\phi_1,\ldots,\phi_n)$ be crystalline anisotropies and let $\Sigma^0:=\bigcup_{i=1}^n\Sigma_i^0$ be an admissible  network such that each $\Sigma_i^0$ consists of $m_i\ge1$ segments $S_1^{0,i},\ldots,S_{m_i}^{0,i}$ and $l_i\in\{0,1\}$ half-lines $L_i^0$. Let $T>0.$ We call a family $\Sigma(t):=\bigcup_{i=1}^n\Sigma_i(t),$ $t\in[0,T),$ of admissible networks a \emph{polycrystalline curvature flow}  starting from $\Sigma^0$ in $[0,T)$ provided $\Sigma(0) = \Sigma^0$ and:
\begin{itemize}
\item[(a)]  $\Sigma(t)$ is parallel to $\Sigma^0,$ i.e., 

\begin{itemize}
\item  each $\Sigma_i(t)$ consists of $m_i$ segments $S_1^i(t),\ldots,S_{m_i}^i(t)$ and $l_i$ half-lines $L_i(t);$

\item each $S_j^i(t)$ is parallel to $S_j^{0,i}$  and $L_i(t)$ and $L_i^0$ lie on the same line; 
\end{itemize}

\item[(b)]  if 
$$
H_j^i(t):= H(S_j^i(t),S_j^{0,i})
$$ 
is the distance vector of the segment $S_j^i(t)$  from $S_j^{0,i},$  then $H_j^i \in C^1((0,T); \R\nu_j^i)\cap C^0([0,T);  \R\nu_j^i)$ and 
\begin{equation*}
\begin{cases}
\dfrac{d}{d t}\,H_j^i(t) = -\phi_i^o(\nu_j^i)\kappa^{\Phi}(S_j^i(t))\nu_j^i & \text{on $(0,T)$}, \\
H_j^i(0) = 0,
\end{cases}
\end{equation*}
where $\nu_j^i$ is the normal to $S_j^{0,i}.$
\end{itemize}
(see Figure \ref{fig:evolution_network1}).
\end{definition}

Parallelness of the flow $\Sigma(t)$ to $\Sigma ^0$ is important feature of the model: later this will be used in the proof the short-time existence of polycrystalline curvature flow of networks (see Theorem \ref{teo:short_time_ccflow}).
 
\begin{remark}
The segment $S_j^i(t)$ moves in the direction of $\nu_j^i$ if and only if $\kappa^\Phi(S_j^i(t)) < 0.$
\end{remark}

\subsection{Computation of crystalline curvature}

In this section we compute the curvature of some networks in the case of a single crystalline anisotropy, see also \cite{BCN:2006}.

Recall that a triplet of vectors $X,Y,Z\in \p B_\phi$ satisfying 
\begin{equation*}
X + Y + Z = 0 
\end{equation*}
is called \emph{admissible} (see Section \ref{subsec:cahn_hoffman_definiton}). 

\begin{figure}[htp!]
\begin{center}
\includegraphics[width = 0.6\textwidth]{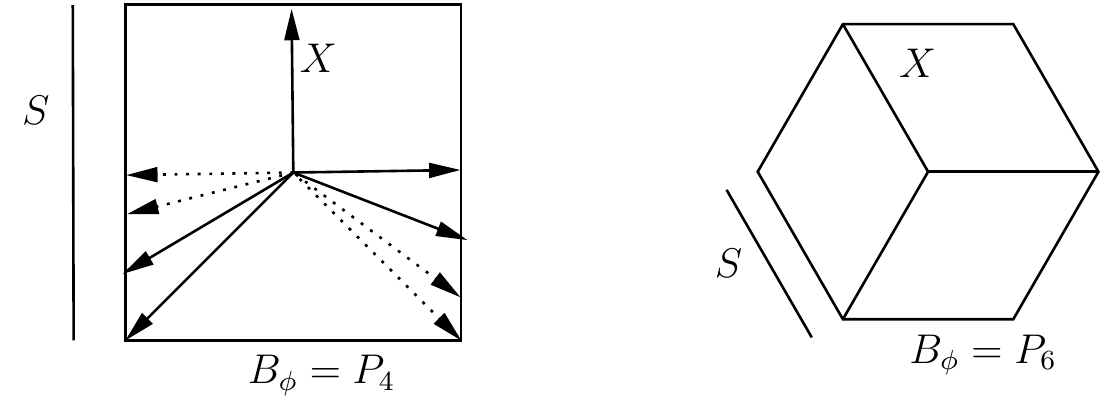}
\caption{\small Admissible triplets for square and hexagonal Wulff-shapes. Note that $P_4$ admits infinitely many pairs $(Y,Z)$ with $X+Y+Z=0.$ However, $P_6$ admits a unique pair for all $X.$}\label{fig:adm_triplets}
\end{center}
\end{figure}

\begin{lemma}[{\cite[Lemma 2.16]{BCN:2006}}]\label{lem:tripletchuk}
Let $\phi$ be any even (not necessarily crystalline) anisotropy in $ \R^2$ and let $X\in\p B_\phi.$ Then there exist two distinct vectors $Y,Z\in\p B_\phi$ such that $(X,Y,Z)$ is admissible. Moreover, if either $B_\phi$ is strictly convex or any segment $S\subset \p B_\phi$ parallel to $X$ satisfies $|S| \le |X|,$ then the pair $Y,Z$ is unique (up to a permutation). Finally, if $\p B_\phi$ contains a segment $S$ satisfying $|S| > |X|,$ then there exist infinitely many unordered pairs $Y,Z\in\p B_\phi$ of disctinct  vectors such that $(X,Y,Z)$ is admissible (see Figure \ref{fig:adm_triplets}).
\end{lemma}

\begin{figure}[htp!]
\begin{center}
\includegraphics[width=0.8\textwidth]{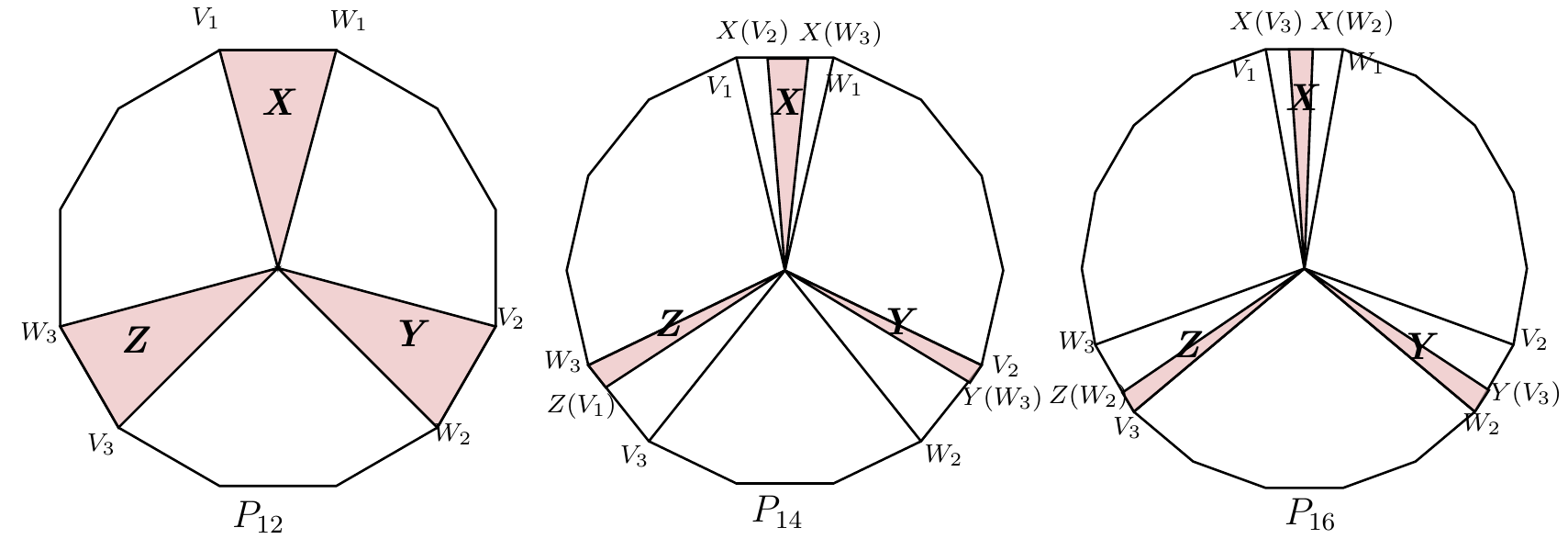}
\caption{\small Wulff shapes $B_\phi = P_{12}$ ($n=6m$), $B_\phi = P_{14}$ ($n=6m-4$) and
$B_\phi = P_{16}$ ($n=6m-2$)
and relative regions of ranging for admissible triplets (filled regions). We observe that the sum of the lengths of three admissible segments in $P_{14}$ and $P_{16}$ is the sidelength of the polygon. Moreover, if any of $X,Y,Z$ belongs to the boundary of its admissible region, then at least one of them falls on a vertex of $B_\phi.$
} 
\label{fig:admissible_triplets}
\end{center}
\end{figure}

Let $\phi$ be a crystalline anisotropy such that $B_\phi$ is a regular polygon with $n\ge6$-vertices (note that $n$ is even) and assume $\phi_i=\phi$ for all $i=1,\ldots,n.$ 
Let us study how admissible triplets look like when one of the vectors $X,Y,Z$ is fixed. By  Lemma \ref{lem:tripletchuk} for any $X$ (resp. $Y$ or $Z$) there exists (up to a permutation) a unique $Y=Y(X)$ and $Z=Z(X)$ (resp. $X(Y)$ and $Z(Y),$ $X(Z)$ and $Y(Z)$) such that $(X,Y,Z)$ is admissible. In view of the symmetry of $B_\phi,$ we can explicitly compute the ranging regions for all admissible triplets (see Figure \ref{fig:admissible_triplets}).

To this aim, let us introduce the following numbers:
\begin{equation}
\theta_n:=
\begin{cases}
\frac{2\pi}{3} & n = 6m,\\
\frac{2\pi}{3}(1 + \frac1n) & n = 6m-4,\\
\frac{2\pi}{3}(1 - \frac1n) & n = 6m-2,
\end{cases}
\quad 
\delta:=\frac{l}{2(1-\cos\theta_n)},\quad 
\bar c = -\frac{1}{2\cos\theta_n},
\label{theta_n_delta_etc}
\end{equation}
\begin{equation}
q_y:=
\begin{cases}
0 & n = 6m,\\
-\bar c\delta & n = 6m - 4,\\
l - \bar c(l-\delta) & n = 6m - 2,
\end{cases}
\qquad\text{and}\qquad 
q_z:=
\begin{cases}
l & n = 6m,\\
\bar c(l-\delta) & n = 6m - 4,\\
l+\bar c\delta & n = 6m - 2,
\end{cases}
\label{qy_qz_defe}
\end{equation}
and the segments 
\begin{equation}\label{seg_ab}
[a,b]:=
\begin{cases}
[0,l] & n = 6m,\\
[\delta, l-\delta] & n=6m -4,\,6m-2.
\end{cases} 
\end{equation}
where $l$ is the sidelength of $B_\phi.$ Then letting 
\begin{equation}\label{def_xyz}
x = |V_1-X|,\quad y = |V_2 - Y|,\quad z = |W_3 - Z| 
\end{equation}
(see Figure \ref{fig:admissible_triplets}), we have 
\begin{equation}
y =  \bar cx + q_y\quad\text{and}\quad z = -\bar cx + q_z 
\label{yx_zxlsasla}
\end{equation}
and in particular, knowing just one among $x,$ $y$ and $z$ we can find the remaining two.

\subsubsection{Crystalline curvature of triods}

Let $\Sigma_1,\Sigma_2,\Sigma_3$ be three polygonal curves  each consisting of one segment $S_i$ and one half-line $L_i$, with a single common vertex, and  let $\Sigma$ be the corresponding network.

We want to compute its crystalline curvature in case $B_\phi$ is a \emph{regular octagon} ($n=8$) of sidelength $l=1$. We parametrize it in such a way that the  $90^\circ$-clockwise rotation of its external unit normal coincide with the tangent to the boundary of $B_\phi$. Then the quantities above are computed as
$$
\theta_8 = \frac{3\pi}{4},\quad \delta = 1 - \frac{1}{\sqrt2},\quad \bar c = \frac{1}{\sqrt2},\quad q_y = - \frac{\sqrt2-1}{2},\quad q_z = \frac12 
$$
and 
$$
[a,b] = \Big[1-\frac{1}{\sqrt2}, \frac{1}{\sqrt2}\Big].
$$
Note that if $x,y,z$ are as in \eqref{def_xyz}, then by \eqref{yx_zxlsasla}
\begin{equation}\label{desc_x_yz222}
y = \frac{x}{\sqrt2} - \frac{\sqrt2-1}{ 2},\qquad z = - \frac{x}{\sqrt2} + \frac12
\end{equation}
so that 
$$
x\in\Big[1-\tfrac{1}{\sqrt2}, \tfrac{1}{\sqrt2}\Big]
\quad\Longrightarrow\quad 
y\in \Big[0,1-\tfrac{1}{\sqrt2}\Big]
\quad\text{and}\quad 
z\in  \Big[0,1-\tfrac{1}{\sqrt2}\Big].
$$
These three segments divide each side of $B_\phi$ into three segments and using Figure \ref{fig:admissible_triplets} the values of admissible triplets $(X,Y,Z)$  are (up to a rotation of $22.5^\circ$) as follows:
\begin{figure}[htp!]
\begin{center}
\includegraphics[width=0.8\textwidth]{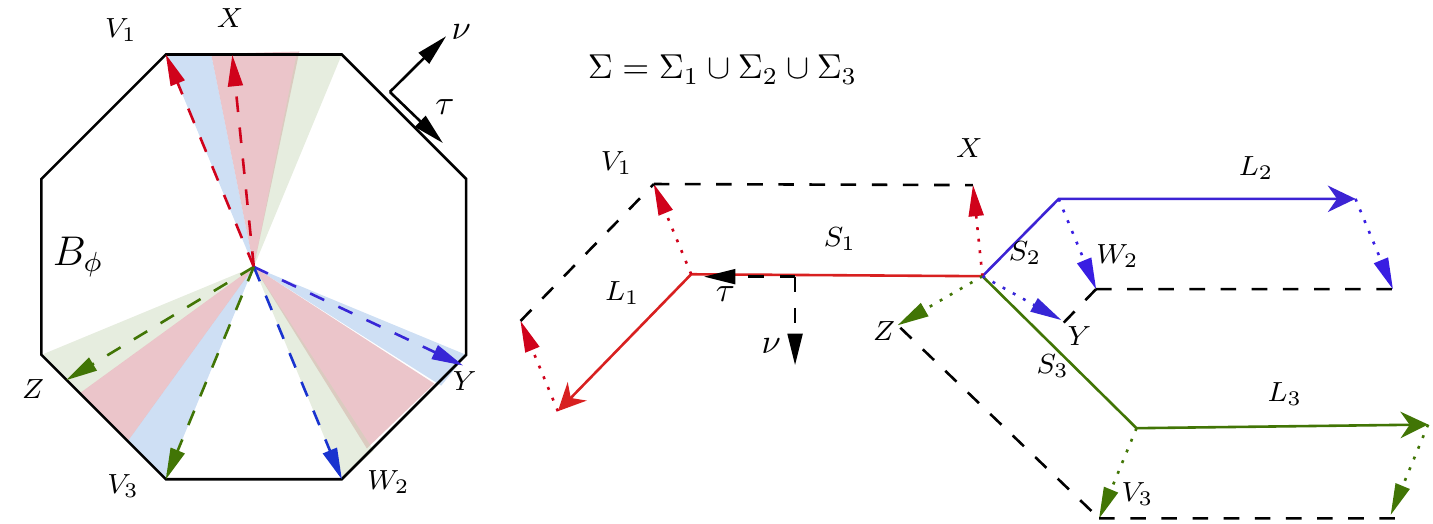} 
\caption{\small A triod and regions for admissible triplets.}
\label{fig:possible_triod}
\end{center}
\end{figure}
For simplicity, assume that our triod is as in Figure \ref{fig:possible_triod}. In this case any Cahn-Hoffman field $N$ is identically equal to $V_1$ on $L_1,$ $W_2$ on $L_2$ and $V_3$ on $L_3.$ According to the figure, in the admissible triplets $(X,Y,Z),$ $X$ must be taken from the ``middle'' region, 
or equivalently,  $x\in [1-\frac{1}{\sqrt2},\frac{1}{\sqrt2}].$ Let us write $a\upuparrows b$ to denote parallel vectors $a$ and $b$ with the same direction. 
Observing  
$(V_1 - X) \upuparrows \tau_{S_1} \upuparrows (-\tau_{B_\phi}(X))$ 
and 
$(W_2 - Y) \upuparrows \tau_{S_2} \upuparrows \tau_{B_\phi}(Y)$
and 
$(V_3 - Z) \upuparrows \tau_{S_3} \upuparrows (-\tau_{B_\phi}(Z))$, from the definition \eqref{def_xyz} of $x,y,z$ we get 
\begin{align*}
\kappa^\Phi(S_1) = & \frac{V_1 - X}{\cH^1(S_1)}\cdot \tau_{S_1} 
= - \frac{x}{\cH^1(S_1)},\\
\kappa^\Phi(S_2) = & \frac{W_2 - Y}{\cH^1(S_2)}\cdot \tau_{S_2}
= \frac{1-y}{\cH^1(S_2)},\\
\kappa^\Phi(S_3) = & \frac{W_3 - Z}{\cH^1(S_3)}\cdot \tau_{S_3}
= - \frac{1-z}{\cH^1(S_3)},
\end{align*}
where in the last two equalities $1$ represents the sidelength of $B_\phi,$ 
and 
$$
\kappa^\Phi(L_i) = 0,\quad i=1,2,3.
$$
Since $\phi^o(\nu_{S_1}) = \phi^o(\nu_{S_3}) = \phi^o(\nu_{S_3}):=\phi_o,$ the functional in \eqref{eNmin} is rewritten as 
\begin{align*}
\int_{\Sigma}[\div_\Sigma N]^2\phi^o(\nu_\Sigma)\,d\cH^1 = &
\sum\limits_{i = 1}^3 \int_{S_i} \phi^o(\nu_{S_i}) [\kappa^\Phi(S_i)]^2\,d\cH^1 \\
= & \phi_o \Big[ \frac{x^2}{\cH^1(S_1)} + \frac{(1-y)^2}{\cH^1(S_2)} + \frac{(1-z)^2}{\cH^1(S_3)}\Big].
\end{align*}
Inserting the representations \eqref{desc_x_yz222} of $y$ and $z$ we get 
\begin{equation}\label{minprob_one_tripod}
\frac{1}{\phi_o}\,\int_{\Sigma}[\div_\Sigma N]^2\phi^o(\nu_\Sigma)\,d\cH^1 =
\alpha x^2 + \beta x + \gamma, 
\end{equation}
where 
\begin{align*}
\alpha:= &  \frac{1}{\cH^1(S_1)} +  \frac{1}{2\cH^1(S_2)} +
 \frac{1}{2\cH^1(S_3)},\\
 \beta:= & \frac{1}{\sqrt2\cH^1(S_3)} - \frac{\sqrt2+1}{\sqrt2\cH^1(S_2)},\\
 \gamma:= & \frac{3+2\sqrt2}{4\cH^1(S_2)} + \frac{1}{4\cH^1(S_3)}.
\end{align*}
Then the minimum problem \eqref{eNmin} reduces to finding
$$
\min\limits_{x\in[1-\frac{1}{\sqrt2},\frac{1}{\sqrt2}]}\,[\alpha x^2 + \beta x + \gamma] 
$$
and the minimizer $x_{\min}$ satisfies 
\begin{equation}\label{xmin_interior9182}
1-\frac{1}{\sqrt2}< x_{\min}  < \frac{1}{\sqrt2} 
\end{equation}
if and only if 
\begin{equation}\label{cond_stabilitysasas}
\frac{\sqrt2-1}{\cH^1(S_1)} + \frac{1}{\sqrt2\cH^1(S_3)} < 
\frac{1}{\cH^1(S_2)} < \frac{\sqrt2}{\cH^1(S_1)} + \frac{\sqrt2}{\cH^1(S_3)}. 
\end{equation}

\begin{remark}
Condition \eqref{xmin_interior9182}  implies that the vector field $N_{\min}$ associated to the network $\Sigma$ given in Figure \ref{fig:possible_triod} at the triple junction belongs to the interior of the admissible region for triplets $(X,Y,Z).$ Thus, any slight modification of $\Sigma$, keeping it $\Phi$-admissible, preserves this ``interiorness'' condition.
We anticipate here that this condition will be used later in the proof of short-times existence of the flow.
\end{remark}

\subsubsection{Crystalline curvature of theta-shaped networks} 
In this subsection we assume that $B_\phi$ is a regular hexagon ($n=6$) of sidelength $l=1$; let $\Sigma$ be a union of two convex hexagons with sides parallel to those of $B_\phi$  and sharing a common side as in Figure \ref{fig:hex_networks}.
\begin{figure}[htp!]
\begin{center}
\includegraphics[width=0.8\textwidth]{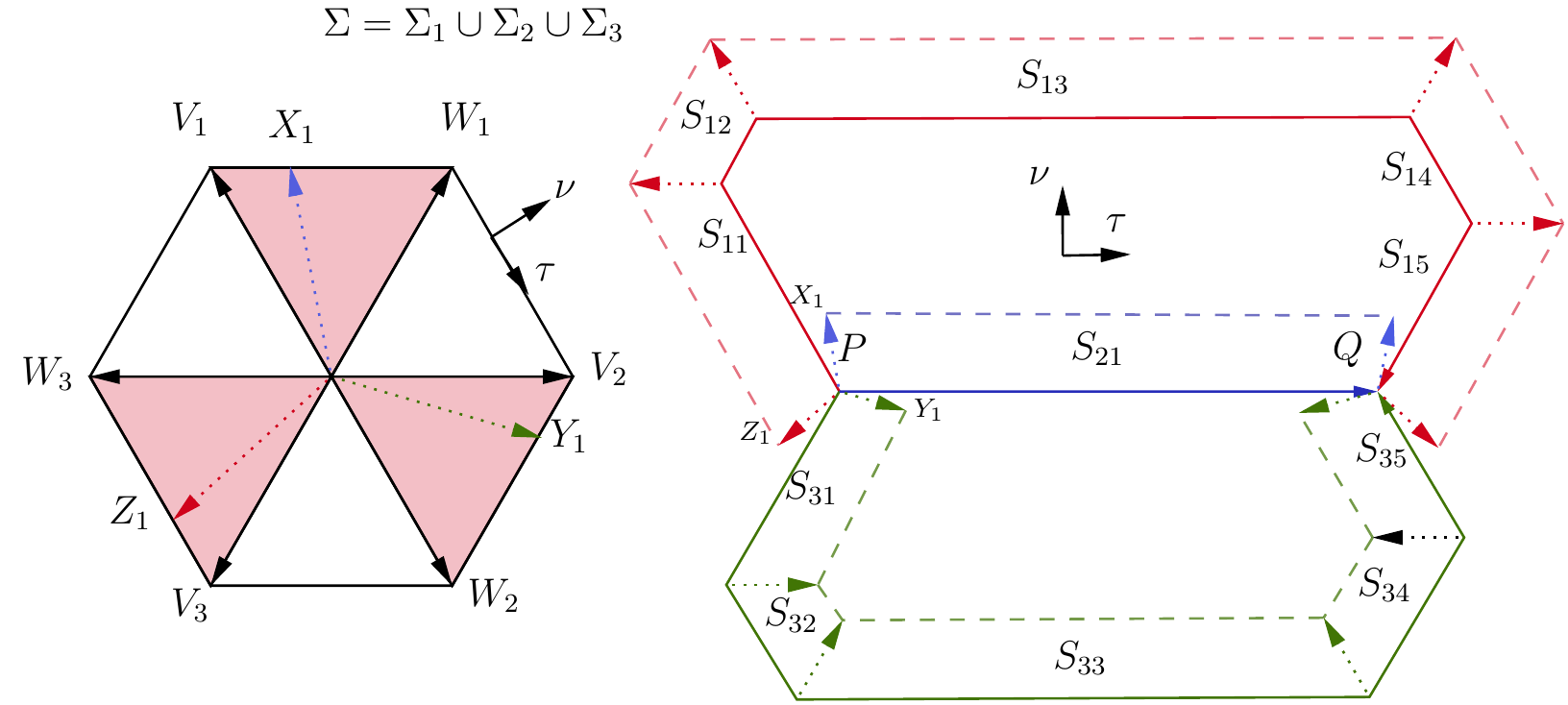}
\caption{ \small Admissible regions for triplets in the hexagon and an admissible $\Theta$-shaped network. Curves $\Sigma_j$ are parametrized from $P$ to $Q$ and $\p B_\phi$ is also parametrized clockwise.}
\label{fig:hex_networks}
\end{center}
\end{figure}
In this case the quantities in \eqref{theta_n_delta_etc}, \eqref{qy_qz_defe} and \eqref{seg_ab} become: $[a,b] = [0,1],$
$$
\theta_6 = \frac{2\pi}{3},\quad \delta = \frac13,\quad \bar c= 1,\quad q_y = 0,\quad q_z = 1
$$
and 
\begin{equation}\label{dnaifeu}
y = x,\quad z = 1 - x. 
\end{equation}
Let $\Sigma_1$ and $\Sigma_3$ be the broken lines consisting of five segments $S_{ij}$ and  $\Sigma_2:=S_{21}$ be a segment (see Figure \ref{fig:hex_networks}). Note that at both triple junctions we have the $120^\circ$-equal angles condition. As we observed above, the crystalline curvatures of $S_{12},S_{13},S_{14}$ and $S_{32},S_{33},S_{34}$ are uniquely defined and equal to 
$$
\kappa^\Phi(S_{ij}) = \frac{1}{\cH^1(S_{ij})},\quad i=1,3,\quad j=2,3,4.
$$

Let $(x_1,y_1,z_1)$ and $(x_2,y_2,z_2)$ be defined as in \eqref{def_xyz} at $P$ and $Q,$ i.e., 
$$
x_i = |V_1 - X_i|,\quad y_i = x_i,\quad z_i = 1-x_i,\quad i=1,2,
$$
where we used \eqref{dnaifeu} in the definitions of $y_i$ and $z_i$.
Then 
\begin{align*}
& \kappa^\Phi(S_{11}) = \frac{W_3 - Z_1}{\cH^1(S_{11})}\cdot \tau_{S_{11}} = \frac{z_1}{\cH^1(S_{11})},\\ 
& \kappa^\Phi(S_{15}) = \frac{Y_2 - V_2}{\cH^1(S_{15})}\cdot \tau_{S_{15}} = \frac{y_2}{\cH^1(S_{15})},\\
& \kappa^\Phi(S_{31}) = \frac{V_2 - Y_1}{\cH^1(S_{31})}\cdot \tau_{S_{31}} = - \frac{y_1}{\cH^1(S_{31})},\\
& \kappa^\Phi(S_{35}) = \frac{Z_2 - W_3}{\cH^1(S_{35})}\cdot \tau_{S_{35}} = - \frac{z_2}{\cH^1(S_{35})},\\
&\kappa^\phi(S_{21}) = \frac{X_2 - X_1}{\cH^1(S_{21})}\cdot \tau_{S_{21}} =  \frac{X_2 - V_1}{\cH^1(S_{21})}\cdot \tau_{S_{21}} -  \frac{X_1 - V_1}{\cH^1(S_{21})}\cdot \tau_{S_{21}}
=
\frac{x_2-x_1}{\cH^1(S_{21})}.
\end{align*}
Since all $\phi^o(\nu_{S_{ij}})$ are equal, denoting their common value by $\phi_o$ we get 
\begin{multline*}
\frac{1}{\phi_o}\int_{\Sigma} [\div_\Sigma N]^2\phi^o(\nu_\Sigma)\,d\cH^1
= \sum\limits_{i\in\{1,3\},\,j \in\{2,3,4\} } \frac{1}{\cH^1(S_{ij})} \\
+ \frac{z_1^2}{\cH^1(S_{11})}
+ \frac{y_2^2}{\cH^1(S_{15})}
+ \frac{y_1^2}{\cH^1(S_{31})}
+ \frac{z_2^2}{\cH^1(S_{35})}
+ \frac{(x_1 - x_2)^2}{\cH^1(S_{21})}.
\end{multline*}
Recalling that $y_i = x_i$ and $z_i = 1- x_i,$ we rewrite the last equality as 
$$
\frac{1}{\phi_o}\int_{\Sigma} [\div_\Sigma N]^2\phi^o(\nu_\Sigma)\,d\cH^1
= \alpha_{11}x_1^2 + 2\alpha_{12}x_1x_2 +  \alpha_{22}x_2^2
+2 \alpha_{1}x_1 + 2\alpha_{2}x_2 +\alpha_0,
$$
where
\begin{align*}
&\alpha_{11} = \frac{1}{\cH^1(S_{11})} + \frac{1}{\cH^1(S_{21})} + \frac{1}{\cH^1(S_{31})},\\
&\alpha_{22} = \frac{1}{\cH^1(S_{15})} + \frac{1}{\cH^1(S_{21})} + \frac{1}{\cH^1(S_{35})},\\
&\alpha_{12} = -\frac{1}{\cH^1(S_{21})},\qquad 
\alpha_{1} = -\frac{1}{\cH^1(S_{11})},\qquad 
\alpha_{2} = -\frac{1}{\cH^1(S_{35})},\\
&\alpha_0 = \sum\limits_{i\in\{1,3\},\,j \in\{1,2,3,4,5\} } \frac{1}{\cH^1(S_{ij})}.
\end{align*}
Thus the minimum problem \eqref{eNmin} reduces to find 
\begin{equation}\label{minprob_two_tripod}
\min\limits_{(x_1,x_2)\in [0,1]^2}\,\Big[\alpha_{11}x_1^2 + 2\alpha_{12}x_1x_2 +  \alpha_{22}x_2^2
+ 2\alpha_{1}x_1 + 2\alpha_{2}x_2 +\alpha_0\Big],
\end{equation} 
whose unique solution is
$$
x_1^{\min}:= \frac{\alpha_{12}\alpha_2 - \alpha_{22}\alpha_1}{\alpha_{11}\alpha_{22} - \alpha_{12}^2},\qquad 
x_2^{\min}:= \frac{\alpha_{12}\alpha_1 - \alpha_{11}\alpha_2}{\alpha_{11}\alpha_{22} - \alpha_{12}^2},
$$
which satisfies $(x_1^{\min},x_2^{\min}) \in (0,1)^2,$ i.e., the values of $N_{\min}$ at both triple junctions belong to the interior of the  admissible regions for triplets (dark regions in Figure \ref{fig:hex_networks}).

\begin{remark}
The minimum problems \eqref{minprob_one_tripod}  and \eqref{minprob_two_tripod} show that, in the case of a single even crystalline  anisotropy, if  $\Sigma$ is an admissible network with $n$-triple junctions, then the minimum problem \eqref{eNmin} is reduced to the minimum problem 
$$
\min\limits_{x_1,\ldots,x_n\in [a,b]} P(x_1,\ldots,x_n),
$$
where 
$$
P(x_1,\ldots,x_n) = \sum\limits_{1\le i \le j \le n} \alpha_{ij}x_ix_j + \sum\limits_{1\le i \le n} \alpha_{i}x_i + \alpha_0
$$
is a quadratic polynomial of $n$-variables and coefficients $\{\alpha_{ij},\alpha_i\}$ depending only on $1/\cH^1(S_k),$ where $S_k$ are segments of $\Sigma$ which end at a triple junction and $\alpha_{ii}>0.$ In other words, the crystalline curvature of a network must ``see'' all triple junctions. This non-locality of crystalline curvature flow makes the problem hard, but at the same time remarkable.
Note that if the minimum $(x_1,\ldots,x_n)$ of $P$ belongs to $(a,b)^n$, then the same holds for all sufficiently small admissible perturbations of $\Sigma.$ Moreover, the minimum is uniquely determined only by the numbers $1/\cH^1(S_k).$
\end{remark}

\subsection{Stable admissible networks}

Figure  \ref{fig:admissible_triplets} and \cite[Definition 2.10]{BCN:2006} encourage  the following definition.

\begin{definition} 
Given crystalline anisotropies $\Phi:=(\phi_1,\ldots,\phi_n),$ a polygonal admissible network $\Sigma=\bigcup_{i=1}^n \Sigma_i$ is called \emph{stable} provided:
\begin{itemize}
\item if $Q$ is a triple junction of curves $\Sigma_{i_1},$ $\Sigma_{i_2}$ and $\Sigma_{i_3},$ then there exists $\epsilon>0$ for which for any $X\in \p B_{\phi_{i_1}}$ with $|X - N_{\min}|_{\Sigma_{i_1}}(Q)|<\epsilon$ there exist $Y\in \p B_{\phi_{i_2}}$ and $Z\in \p B_{\phi_{i_3}}$ such that 
$$
|Y - N_{\min}|_{\Sigma_{i_2}}(Q)| + |Z - N_{\min}|_{\Sigma_{i_3}}(Q)| < \epsilon\quad\text{and}\quad X+Y+Z = 0;
$$ 

\item $N_{\min}|_{\Sigma_{i_j}}(Q)$ is not a vertex of $B_{\phi_{i_j}}.$

\end{itemize}

\end{definition}

In other words, a network is stable if and only if the minimal Cahn-Hoffman vector field at every triple junction lies in the interior of the corresponding admissible regions of triplets. For instance, for the Wulff shapes depicted in Figure \ref{fig:three_bad_hex} any network satisfying $N_{\min} = X$ at a triple junction is not stable.
As we have seen in the case of a single anisotropy $\phi,$ whose Wulff shape is a regular polygon, any admissible $\Sigma$ is stable if and only if at each triple junction $N_{\min}$ is not a vertex of $B_\phi$ (see Figure \ref{fig:admissible_triplets}).  In particular, the theta-shaped network in Figure \ref{fig:hex_networks} and the triod in Figure \ref{fig:possible_triod} (provided \eqref{cond_stabilitysasas} holds) are stable.  

\begin{figure}[htp!]
\begin{center}
\includegraphics[width=0.9\textwidth]{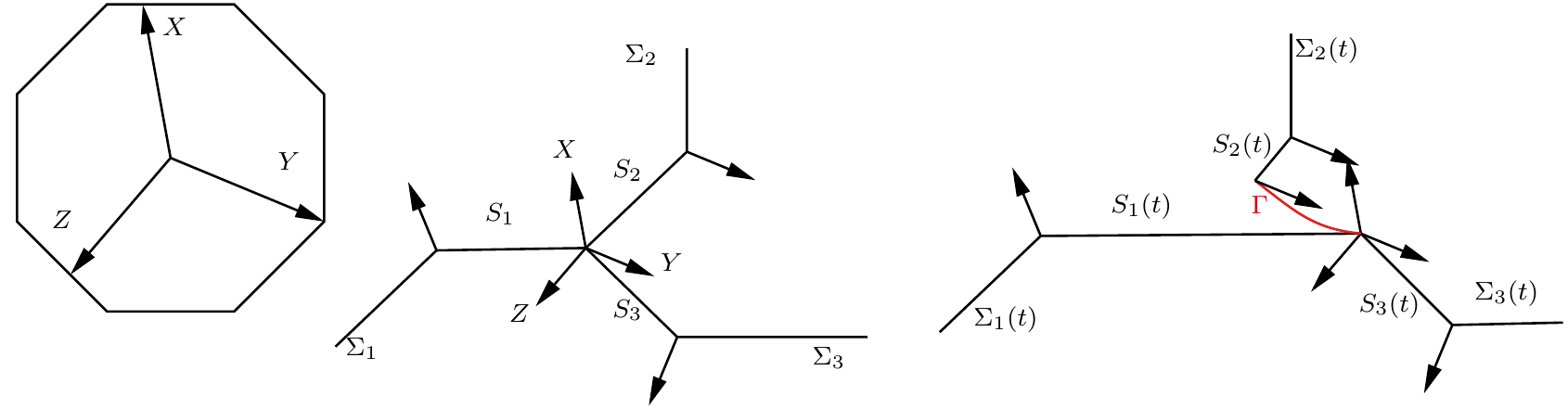}
\caption{\small Appearance of a $\Phi$-zero curvature curve in an unstable network. Notice that in this situation $S_2$ has zero crystalline curvature, but $S_1$ and $S_3$ have non-zero crystalline curvature, therefore, in the evolution they move, whereas $S_2$ does not move. Thus, to continue the motion, the network  creates a ``curved'' part, which does not fall within our simplified definition of polycrystalline curvature flow.}
\end{center}
\end{figure}

\begin{remark}
For any admissible $\Sigma:=\bigcup_{i=1}^n\Sigma_i$ let 
\begin{equation}\label{msaushuah}
\delta_\Sigma:=\min\,|\restriction {N_{\min}^{\Sigma}}|{\Sigma_i}(Q)- \restriction {N}|{\Sigma_i}(Q)|,
\end{equation}
where the minimum is taken over all triple junctions $Q$ and all Cahn-Hoffman vector fields $N\in\cN^\Sigma$ such that at least one $\restriction {N}|{\Sigma_i}(Q)$ belongs to the boundary of the admissible region or to a vertex of $B_{\phi_i}.$ Then $\Sigma$ is stable if and only if $\delta_\Sigma>0.$ Therefore, in view of Lemma \ref{lem:curvature_parallel_network}, as we have observed above in the octagon and hexagon examples, a slight (still admissible) perturbation of a stable network is again stable. An interesting phenomena may occur in the unstable (i.e. in the not stable) case: in this case a slight perturbation of the network either becomes stable or a new zero $\Phi$-curvature curve/segment may start to grow from the triple junction. A discussion on such phenomena can be found in \cite{BCN:2006} for a single anisotropy.
\end{remark}

\begin{theorem}[Local existence and uniqueness]\label{teo:short_time_ccflow}
Let $\Phi:=\{\phi_i\}_{i=1}^3$ be crystalline anisotropies and $\Sigma^0:=\cup_{i=1}^3\Sigma_i^0$ be a stable polygonal network having a single triple junction, where each $\Sigma_i^0$ consists of a single segment $S_i$ and a half-line $L_i$, oriented starting from the triple junction. Then there exist $T>0$ and a unique polycrystalline curvature flow $\{\Sigma(t)\}_{t\in[0,T]}$ of admissible stable networks starting from $\Sigma^0.$ 
\end{theorem}

\begin{proof}
Let $S_i$ (resp. $L_i$), $i=1,2,3,$ be the segments (resp. half-lines) forming a triple junction, oriented from the triple junction, and let $\theta_i$ be the angle between $S_j$ and $S_k,$ $i\ne j\ne k\ne i,$ so that 
\begin{equation}\label{nosnoasnoasn}
\cos \theta_i=\tau_{S_j}\cdot \tau_{S_k},\quad \theta_1 + \theta_2 + \theta_3 = 2\pi. 
\end{equation}
Let us denote by $\gamma_i\in(0,\pi)$ the angle between $S_i $ and $L_i$ at their junction.
\smallskip

{\it Step 1.} Given $\rho>0,$ let $\cG_\rho$ be the collection of all admissible networks $\Sigma$ parallel to $\Sigma^0$ such that $d(\Sigma,\Sigma^0)\le \rho.$ Let us show that there exists $\rho_0>0$ depending only on $\Sigma^0$ such that any network $\Sigma\in \cG_{\rho_0}$ is stable.

For $\rho>0$ and $\Sigma\in \cG_\rho,$ let $\delta_{\Sigma}\ge0$ be given by \eqref{msaushuah}. Since $\Sigma^0$ is stable, we have $\delta_{\Sigma^0}>0.$ Assume that there exists a sequence $\{\Sigma(k)\}$ of unstable networks parallel to $\Sigma^0$ such that $d(\Sigma(k),\Sigma^0)<1/k$ for any $k\ge1.$ Then for $i=1,2,3,$ by \eqref{conve_cahnhoffman} 
$$
\restriction {N_{\min}^{\Sigma(k)}}|{S_i(k)}\big(\zeta(S_i(k),S_i,x)\big) \to \restriction {N_{\min}^{\Sigma^0}}|{S_i}(x)
$$
uniformly in $x\in S_i$, where 
$\zeta(S_i(k),S_i,\cdot):S_i\to S_i(k)$ is a linear bijection of $S_i$ to the corresponding parallel segment $S_i(k)\subset\Sigma(k).$ Thus, $0=\delta_{\Sigma(k)} \ge \delta_{\Sigma^0}/2>0$ for all large $k,$ a contradiction.

{\it Step 2.} Let $T_1,T_2,T_3$ (resp. $\bar T_1,\bar T_2,\bar T_3$) be segments, forming a triple junction and oriented from the triple junction, such that $T_i||\bar T_i$ for $i=1,2,3.$ If 
$$
h_i:=H(\bar T_i,T_i)\cdot \nu_{T_i},\quad i=1,2,3,
$$
then 
\begin{equation}\label{ndezf7fbe}
h_1\sin\omega_1 + h_2\sin\omega_2 + h_3\sin\omega_3 = 0,
\end{equation}
where $\omega_i$ is the angle between $T_j$ and $T_k,$ $i\ne j\ne k\ne i.$

Assume that we are as in the situation of Figure \ref{fig:x1x2x399}, i.e., $h_3\ge0\ge h_1,h_2$ and let $x_i = |h_i|$ for $i=1,2,3$.

\begin{figure}[htp!]
\begin{center}
\includegraphics[width = 0.5\textwidth]{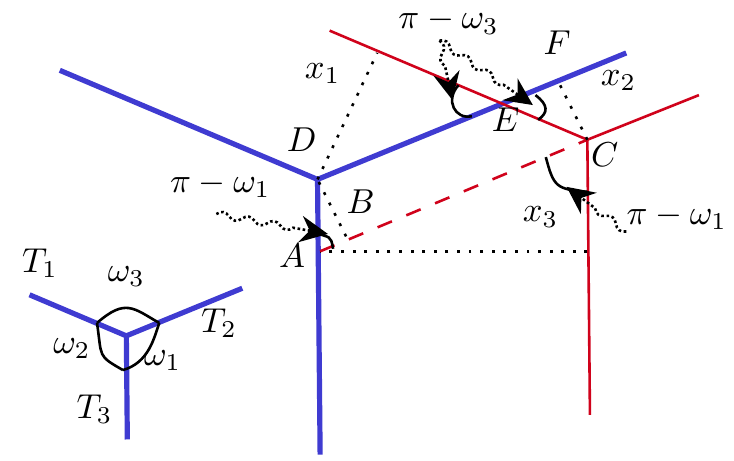}
\caption{\small Expressing $x_3$ with $x_1$ and $x_2.$}
\label{fig:x1x2x399}
\end{center}
\end{figure}
\noindent
Then 
$$
\cH^1([AC]) = \frac{x_3}{\sin(\pi-\omega_1)} = \frac{x_3}{\sin\omega_1}.
$$
On the other hand, since 
\begin{align*}
\cH^1([AC]) = & \cH^1([AB]) + \cH^1([DE]) + \cH^1([EF]) \\
= & -\frac{x_2}{\tan\omega_1} + \frac{x_1}{\sin\omega_3} - \frac{x_2}{\tan\omega_3}  
=  \frac{x_1}{\sin\omega_3} - \frac{x_2\sin(\omega_1+\omega_3)}{\sin\omega_1\sin\omega_3}
\end{align*}
and $\omega_1 + \omega_2 +\omega_3=2\pi,$ we have 
$$
\frac{x_3}{\sin\omega_1} = \frac{x_1}{\sin\omega_3} + \frac{x_2\sin\omega_2}{\sin\omega_1\sin\omega_3}
$$
and \eqref{ndezf7fbe} follows. The proof in the other cases is similar.

{\it Step 3.} Let $h_1,h_2,h_3$ be real numbers such that $|h_i|\le \rho_0$ and 
\begin{equation}\label{xaocsbvu}
\sum\limits_{i=1}^3 h_i\sin\theta_i = 0. 
\end{equation}
Then there exists a unique $\Sigma\in \cG_{\rho_0}$ such that 
$$
h_i = H_i(S_i,S_i^0)\cdot \nu_{S_i^0},\quad i=1,2,3.
$$
Indeed, using $h_1$ and $h_2$ we can construct a unique network $\Sigma$ parallel to $\Sigma^0$ and $h_i=H_i(S_i,S_i^0)\cdot \nu_{S_i^0}$ for $i=1,2.$ Let $h_3':=H(S_i,S_i^0)\cdot \nu_{S_i^0}.$ 
By Step 2 we know 
$$
h_1\sin\theta_1 + h_2\sin\theta_2 +h_3'\sin\theta_3= 0.
$$
Then \eqref{xaocsbvu} implies $h_3=h_3'$ and hence, $\Sigma\in \cG_{\rho_0}.$

{\it Step 4.} Let us study some properties of the minimizing Cahn-Hoffman vector field  $N_\star:=N_{\min}^\Sigma$ of $\Sigma\in \cG_{\rho_0}.$ 

By Remark \ref{rem:prop_curvature812} $N_\star$ is uniquely defined at the endpoints of each $L_i,$ and coincide with $N_\star^0:=N_{\min}^{\Sigma^0}.$ Thus, we only need to care at the triple junction of $S_1,S_2,S_3$. Writing $S_i:=[A_0A_i]$ and using Lemma \ref{lem:curvature_parallel_network} (b) we find that $N_{\star}$ minimizes the functional 
$$
\cF_\Sigma(N):=\sum\limits_{i=1}^3\frac{ \phi_i^o(\nu_{S_i}) }{\cH^1(S_i)}\,\Big|[\restriction N|{S_i}(A_i) - \restriction N|{S_i}(A_0)]\cdot \tau_{S_i} \Big|^2,\quad N\in\cN^\Sigma,
$$
where $\phi_{i}$ is the anisotropy, corresponding to $\Sigma_i.$ Since $A_i$ is the endpoint of the half-line $L_i,$ as we observed above $N(A_i) = N_\star(A_i) = N_\star^0(A_i^0),$ where we set $S_i^0:=[A_0^0A_i^0].$ Thus, recalling $\nu_{S_i} = \nu_{S_i^0}$ and $\tau_{S_i} = \tau_{S_i^0},$ we get
\begin{equation}\label{starstarstara}
\cF_\Sigma(N):=\sum\limits_{i=1}^3\frac{ \phi_{i}^o(\nu_{S_i^0}) }{\cH^1(S_i)}\,\Big|a_i - x_i\Big|^2,
\end{equation}
where
$$
a_i:=[\restriction {N_\star^0}|{S_i^0}(A_i^0) - \restriction {N_\star^0}|{S_i^0}(A_0)]\cdot \tau_{S_i^0} = \kappa^\Phi(S_i^0)\cH^1(S_i^0)
$$
and
$$
x_i:=[\restriction N|{S_i}(A_0) - \restriction {N_\star^0}|{S_i^0}(A_0^0)]\cdot \tau_{S_i^0}.
$$
By the choice of $\rho_0$ and the stability of $\Sigma,$ we may consider only those $N\in\cN^\Sigma$ such that  $\restriction N|{S_i}(A_0)$ and $\restriction {N_\star^0}|{S_i^0}(A_0^0)$ lie in the same side of the Wulff shape $B_{\phi_{k_i}}$, parallel to $\tau_{S_i^0},$ and therefore
$$
\restriction N|{S_i}(A_0) =  \restriction {N_\star^0}|{S_i^0}(A_0^0) +  x_i\tau_{S_i^0}.
$$
Now by the balance condition at $A_0$ we get 
$$
\sum\limits_{i=1}^3 x_i\tau_{S_i^0} = 0,
$$
or equivalently, by \eqref{nosnoasnoasn}
$$
x_i = -x_j \cos\theta_k - x_k \cos\theta_j,\quad i\ne j\ne k\ne i.
$$
These equalities immediately imply 
$$
\frac{x_1}{\sin\theta_1} = \frac{x_2}{\sin\theta_2} = \frac{x_3}{\sin\theta_3}. 
$$
Therefore, we have only one independent variable, i.e., as in the case of a single crystalline anisotropy, knowing only one value of $\restriction N|{S_i}$ we uniquely determine the admissible triplet. Letting 
\begin{equation}\label{hshahahashs}
x_2 = x_1\,\frac{\sin \theta_2}{\sin\theta_1}\quad\text{and}\quad x_3 = x_1\,\frac{\sin \theta_3}{\sin\theta_1}, 
\end{equation}
we can rewrite $\cF_\Sigma(N)$ in \eqref{starstarstara} as a quadratic function of $x_1:$
$$
\cF_\Sigma(N) = \alpha_2 x_1^2 + \alpha_1x_1 + \alpha_0:=f(x_1),
$$
where 
\begin{align*}
\alpha_2 = & \frac{1}{\sin^2\theta_1} \sum\limits_{i=1}^3 \frac{\phi_i^o(\nu_{S_i^0})\sin^2\theta_i}{\cH^1(S_i)},\\
\alpha_1 = & -\frac{2}{\sin\theta_1} \sum\limits_{i=1}^3 \phi_i^o(\nu_{S_i^0}) \kappa^\Phi(S_i^0)\,\frac{\cH^1(S_i^0)\sin\theta_i}{\cH^1(S_i)},\\
\alpha_0 = & \sum\limits_{i=1}^3  \frac{\phi_i^o(\nu_{S_i^0})}{\cH^1(S_i)}\, [\kappa^\Phi(S_i^0)\cH^1(S_i^0)]^2.
\end{align*}

Since $\cF_\Sigma$ has a unique minimizer $N=N_{\star}$ and $\Sigma$ is stable, $f$ has a unique minimizer $x_1^\star:=-\frac{\alpha_1}{2\alpha_0}.$ Thus, uniquely defining $x_2^\star$ and $x_3^\star$ using the relations \eqref{hshahahashs}, we get
$$
\restriction {N_\star}|{S_i}(A_0) =  \restriction {N_\star^0}|{S_i^0}(A_0^0) +  x_i^\star \tau_{S_i^0}.
$$
Therefore, by the definition of the crystalline curvature,
\begin{equation}\label{mucci_mucci}
\kappa^{\Phi}(S_i) = \frac{[\restriction {N_\star}|{S_i}(A_i) - \restriction {N_\star}|{S_i}(A_0)]\cdot \tau_{S_i^0}}{\cH^1(S_i)} = \frac{\cH^1(S_i^0)\kappa^\Phi(S_i^0) - x_i^\star}{\cH^1(S_i)}. 
\end{equation}
By uniqueness, $N_\star = N_\star^0$ if $\Sigma = \Sigma^0,$ and thus, in this case we should have $x_i^\star = 0.$ Then by the explicit expression of $x_1^\star$ we get 
\begin{equation}\label{woho_whosho}
\sum\limits_{i=1}^3 \phi_i^o(\nu_{S_i^0})\kappa^\Phi(S_i^0)\sin\theta_i = 0. 
\end{equation}
Note that this condition on curvatures is analogous to the smooth case \eqref{no_tangency}.

Similarly, slightly perturbing $\Sigma$ and repeating the same arguments above we find that \eqref{woho_whosho} holds also with $\Sigma$ in place of $\Sigma^0,$ i.e., any admissible stable network $\Sigma,$ parallel to $\Sigma^0,$ satisfies this curvature-balance condition. This condition will be important in the sequel.

{\it Step 5.} By Proposition \ref{prop:height_vs_length} we can compute 
$$
\cH^1(S_i) = \cH^1(S_i^0) + g_i(h_1,h_2,h_3)
$$
for some positively one-homogeneous Lipschitz function $g_i.$ Inserting this in \eqref{mucci_mucci} we get
$$
\kappa^\Phi(S_i) = \kappa^\Phi(S_i^0) + G_i(h_1,h_2,h_3),
$$
where $G_i$ is a Lipschitz function, satisfying $G_i(0,0,0) = 0.$

{\it Step 6.} Given $T>0,$ let 
$
\cS_T
$ be the space of all triplets $h:=(h_1,h_2,h_3)$ of functions $h_i\in C([0,T])$ satisfying 
\begin{equation}\label{nedbhfjv45}
h_1(t)\sin\theta_1 + h_2(t)\sin\theta_2 + h_3(t)\sin\theta_3 =0,\quad i\in[0,T],
\end{equation}
and 
\begin{equation}\label{max_heigh_boundaoda}
h_i(0) = 0,\quad \max_i \|h_i\|_\infty\le \rho_0. 
\end{equation}
Consider the map $F_h=(F_h^1,F_h^2,F_h^3),$ where
$$
F_h^i(t):=\phi_i^o(\nu_{S_i^0})\,\int_0^t \Big(\kappa^\Phi(S_i^0) + G_i(h_1(s),h_2(s),h_3(s))\Big)ds,\quad t\in[0,T].
$$
We claim that $F_h$ has a fixed point in $\cS_T$ provided that 
$$
T \le \frac{\rho_0}{1 + \max_i \Big(\|\phi_i^o\|_\infty [|\kappa^\Phi(S_i)| + 3\ell_i\rho_0 ]\Big)},
$$
where  $\ell_i>0$ is the Lipschitz constant of $G_i.$
Indeed, note that $\cS_T$ is convex and closed. Let us show that $F_h\in \cS_T$ for any $h\in \cS_T.$ Clearly, $F_h(0) = 0.$ By the definition of $G_i,$
$$
|G_i(h_1,h_2,h_3)| \le \ell_i (|h_1| + |h_2|+|h_3|) \le 3\ell_i\rho_0.
$$
Therefore, by the choice of $T,$
$$
\|F_h^i\|_\infty \le \|\phi_i^o\|_\infty (|\kappa^\Phi(S_i)| + 3\ell_i\rho_0) T < \rho_0.
$$
Next we show 
\begin{equation}\label{no_ya_igrayu_eturol}
\sum\limits_{i=1}^3 F_h^i(t)\sin\theta_i = 0\quad\text{for any  $t\in [0,T]$}. 
\end{equation}
By assumptions \eqref{nedbhfjv45}-\eqref{max_heigh_boundaoda} and Step 3, for each $s\in[0,T]$ there exists a unique network $\Sigma(s)\in \cG_{\rho_0}$ satisfying $h_i(s) = H_i(S_i(s),S_i^0)\cdot \nu_{S_i^0}$.
Then by Step 5 its curvature is given by
\begin{equation}\label{ancsdvubuc}
\kappa^\Phi(S_i(s)) = \kappa^\Phi(S_i^0) + G_i(h_1(s),h_2(s),h_3(s)) 
\end{equation}
and by Step 4, it satisfies the curvature-balance condition
$$
\sum\limits_{i=1}^3 \phi_i^o(\nu_{S_i^0}) \kappa^\Phi(S_i(s))\sin\theta_i = 0.
$$
From this and \eqref{ancsdvubuc} we deduce 
$$
\sum\limits_{i=1}^3 \phi_i^o(S_i^0) [\kappa^\Phi(S_i^0) + + G_i(h_1(s),h_2(s),h_3(s)) ] \sin\theta_i = 0,\quad s\in[0,T],
$$
and hence integrating this equality we obtain \eqref{no_ya_igrayu_eturol}.

Next let us prove that the set $\{F_h\}_{h\in\cS_T}$ is compact in $(C([0,T]))^3$. Indeed, by the choice of $T,$ this set is equibounded in $(C([0,T]))^3.$ Moreover, as we have seen above
$$
\|[F_h^i]'\|_\infty \le |\kappa^\Phi(S_i^0)| + 3\ell_i\rho_0.
$$
Thus,  $\{F_h\}_{h\in\cS_T}$ is also equicontinuous in $(C[0,T])^3.$ Therefore, by the Arzela-Ascoli theorem, it is compact. Now the Schauder fixed point theorem implies the existence of $\bar h\in\cS_T$ such that 
$F_{\bar h} = \bar h.$ A bootstrap argument shows that $\bar h\in C^1([0,T])$ and hence, $\bar h' = F_{\bar h}',$ i.e.,
$$
\bar h_i'(t) =\phi_i^o(S_i^0) \Big(\kappa^\Phi(S_i^0) + G_i(\bar h_1(t),\bar h_2(t),\bar h_3(t)\Big),\quad i=1,2,3.
$$
As we observed earlier, the right-hand side of this equality is the (multiple of the) crystalline curvature of the unique network $\bar\Sigma(t)\in\cG_{\rho_0}$ such that 
$$
\bar h_i(t) = H_i(\bar S_i(t), S_i^0)\cdot \nu_{S_i^0},\quad t\in[0,T].
$$
In particular, by Step 1 each $\bar\Sigma(t)$ is stable. Since $h_i(0) = 0,$ we have $\bar\Sigma(0)=\Sigma^0$, and therefore $\{\bar\Sigma(\cdot)\}$ is the crystalline curvature flow of stable networks starting from $\Sigma^0$ (in the sense of Definition \ref{def:curvature_flow}).

Finally, let us show the uniqueness of the flow. Let $\{\bar \Sigma(\cdot)\}$ and $\{\hat \Sigma(\cdot)\}$ be two different flows  in $[0,T]$ starting from $\Sigma^0,$ and let $\bar h:=(\bar h_1,\bar h_2,\bar h_3)$ and  $\hat h:=(\hat h_1,\hat h_2,\hat h_3)$ be the corresponding signed distances. Since both solve the same ODE of Definition \ref{def:curvature_flow}, we have 
$$
\bar h = F_{\bar h}\quad\text{and}\quad \hat h = F_{\hat h}.
$$
Thus, by the Lipschitzianity of $G_i$ we get
$$
|\bar h(t) - \hat h(t)| \le  \max_i \ell_i\|\phi_i^o\|_\infty \|\bar h - \hat h\|_\infty T\quad\text{for any $t\in[0,T]$}.
$$
Thus, by the choice of $T,$ we get 
$$
0< \|\bar h - \hat h\|_\infty \le  \max_i \ell_i\|\phi_i^o\|_\infty \|\bar h - \hat h\|_\infty T < \|\bar h - \hat h\|_\infty,
$$
a contradiction. 
\end{proof}


\section{Appendix}

In this appendix we prove two lemmas used in the proof of Theorem \ref{teo:existence_special_mcf}.

\begin{lemma}[\textbf{H\"older continuity of the composition}]\label{lem:Holder_norm_of_composition}
Let $\Omega\subset {\mathbb R}^n$ be a bounded open set, $\rho\in(0,1]$ and let $G\in C^{1+\rho}(\overline  \Omega).$ Then
\begin{align}\label{Holder_G_gute_2d}
[G(b_1) - G(b_2)]_{\rho'} \le \|\nabla G\|_\infty [b_1-b_2]_{\rho'} 
+ [\nabla G]_\rho \max\{[b_1]_\rho,[b_2]_{\rho'}\}\|b_1- b_2\|_\infty^\rho
\end{align}
for any $b_1,b_2\in C^{\rho'}([0,1];\cl{\Omega})$ with $\rho'\in(0,\rho],$ where $[\cdot]_\rho$ is the H\"older seminorm.
\end{lemma}

\begin{proof}
First we assume that $n=1$ and $\Omega$ is a bounded interval of ${\mathbb R}.$ Then for any $b_1,b_2\in C^{\rho'}([0,1];\cl{\Omega})$  and $x,y\in[0,1],$ $x\ne y,$ we have
\begin{align}\label{hurmatli_muxlislar}
G(b_1(x)) - G(b_2(x)) - G(b_1(y)) + G(b_2(y))
= 
\int_{b_1(y)}^{b_1(x)} G'(t)dt - \int_{b_2(y)}^{b_2(x)} G'(t)dt.
\end{align}
Since 
$$
\int_a^b G'(t)dt = (b-a)\int_0^1 G'(a + t(b-a))d t,
$$
we can rewrite \eqref{hurmatli_muxlislar} as 
\begin{align*}
& G(b_1(x)) - G(b_2(x)) - G(b_1(y)) + G(b_2(y)) \\
= & (b_1(x) - b_1(y)) \int_0^1 \Big[G'(b_1(y) + t[b_1(x) - b_1(y)]) - G'(b_2(y) + t[b_2(x) - b_2(y)])\Big] dt \\
& + (b_1(x) - b_1(y) - b_2(x) + b_2(y)) 
\int_0^1 G'(b_1(y) + t[b_1(x) - b_1(y)]) dt.
\end{align*}
Hence, using $G'\in C^\rho(\overline\Omega)$ we get

\begin{align*}
& \frac{|G(b_1(x)) - G(b_2(x)) - G(b_1(y)) + G(b_2(y))|}{|x-y|^{\rho'} }\\
\le & [b_1]_{\rho'} [G']_\rho \int_0^1 \Big|(1-t)|b_1(y) - b_2(y)| + t|b_1(x) - b_1(y)|\Big|^\rho dt +[b_1 - b_2]_{\rho'} \|G'\|_\infty.
\end{align*}
Therefore,
\begin{equation}\label{one_dime_estimos}
[G(b_1) - G(b_2)]_{\rho'} \le \|G'\|_\infty[b_1-b_2]_{\rho'} + [G']_\rho\max\{[b_1]_{\rho'},[b_2]_{\rho'}\}\,\|b_1-b_2\|_\infty^\rho. 
\end{equation}

Now consider the case $n>1$ and let $b_i = (b_i^1,\ldots,b_i^n).$ Then by the triangle inequality we have
\begin{align*}
[G(b_1) - G(b_2)]_{\rho'} \le & [G(b_1^1,b_2^1\ldots,b_1^n) - G(b_2^1,b_1^1\ldots,b_1^n)]_{\rho'}
+ \ldots \\
& + [G(b_2^1,\ldots, b_2^{n-1},b_1^n) - G(b_2^1,\ldots, b_2^{n-1},b_2^n)]_{\rho'}.
\end{align*}
By \eqref{one_dime_estimos}
\begin{multline*}
[G(\cdots, b_1^i,\cdots) - G(\cdots,b_2^i,\cdots)]_{\rho'} \\
\le \sum\limits_{i=1}^n\|\nabla_iG\|_\infty [b_1^i - b_2^i]_{\rho'}
+ [\nabla_iG]_\rho \max\{[b_1^i]_{\rho'},[b_2^i]_{\rho'}\}\|b_1^i-b_2^i\|_\infty^{\rho},
\end{multline*}
where the corresponding variables in place of $\cdots$ in $G(\cdots, b_1^i,\cdots)$ and $  G(\cdots,b_2^i,\cdots)$ are the same and $\nabla_if:=\frac{\p f}{\p x_i}.$
Therefore, 
$$
[G(b_1) - G(b_2)]_{\rho'} \le \|\nabla G\|_\infty \sum\limits_{i=1}^n[b_1^i-b_2^i]_{\rho'} + [\nabla G]_\rho \sum\limits_{i=1}^n\max\{[b_1^i]_{\rho'},[b_2^i]_{\rho'}\}\|b_1^i - b_2^i\|_\infty^\rho
$$
and \eqref{Holder_G_gute_2d} follows.
\end{proof}

\begin{lemma}\label{lem:holder_estimates_for_ai}
Let $\psi$ be a $C^{2+\alpha}$-function defined in a tubular neighborhood of the unit circle $\S^1 \subset\R^2$ such that 
$$
0< m \le \min\limits_{\nu\in\S^1}\,\psi(\nu) \le \max\limits_{\nu\in\S^1}\,\psi(\nu)\le \frac{1}{m}
$$
for some $m\in(0,1].$ For $T>0$ and $m_1,m_2>0$, let $X_{T,m_1,m_2}$ be the subset of all $v\in C_T^{\frac{\alpha}{2},\alpha}$ such that 
$$
\min\limits_{t\in[0,T], x\in[0,1]} |v(t,x)| \ge m_1\quad\text{and}\quad \|v\|_{C_T^{\frac{\alpha}{2},\alpha}} \le m_2.
$$
For $v\in X_{T,m_1,m_2}$ define 
$$
a(v):= |v|^{-2}\psi\Big(\frac{v}{|v|}\Big).
$$
Then for any $v',v'' \in X_{T,m_1,m_2}$  
\begin{align}
& \| a(v') - a(v'')\|_\infty \le L\,\|v'-v''\|_\infty, \nonumber\\ 
& [a(v') - a(v'')]_{\alpha,x} \le L \Big[\|v'-v''\|_\infty 
+[v'-v'']_{\alpha,x}\Big],\label{a1_holder_zest01}\\
& [a(v') - a_1(v'')]_{\alpha/2,t} \le L \Big[\|v'-v''\|_\infty  + [v'-v'']_{\alpha/2,t}\Big], \label{a1_holder_test01}
\end{align}
and
\begin{align}\label{ai_holder_est009}
\|a(v')\|_{C_T^{\alpha,\frac{\alpha}{2}}} \le L, 
\end{align}
where $L$ depends (continuously) only on $m_1,$ $m_2,$ $\|\psi\|_\infty,$ $\|\nabla\psi\|_\infty$ and $\|\nabla^2\psi\|_\infty.$ 
\end{lemma}

In the proof of Theorem \ref{teo:existence_special_mcf} we apply this lemma with functions whose space-derivative $v_x$ belongs to $X_{T,\delta/2,M}$ and $\psi = \beta_j.$

\begin{proof}
Note that $\xi\in{\mathbb R}^2\mapsto |\xi|^{-n},$ $n\in\N,$ is Lipschitz continuous in $\{|\xi|\ge m_1\}.$ Indeed, 
$$
\begin{aligned}
\big||\xi_1|^{-n}  - |\xi_2|^{-n}\big| \le & \frac{|\xi_1 - \xi_2|}{|\xi_1|^n
|\xi_2|^n}\sum\limits_{i=0}^{n-1} |\xi_1|^i|\xi_2|^{n-1-i} \\ 
\le &
\frac{n|\xi_1 - \xi_2|}{|\xi_1|^n
|\xi_2|^n}\, \max\{|\xi_1|^{n-1},|\xi_2|^{n-1}\}\\
\le & \frac{n|\xi_1 - \xi_2|}{\min\{|\xi_1|^n,|\xi_2|^n\} \max\{|\xi_1|,|\xi_2|\}}
\end{aligned}
$$
so that 
\begin{equation*}
\big||\xi_1|^{-n}  - |\xi_2|^{-n}\big| \le \frac{n}{ m_1 ^{n+1}}\,|\xi_1 - \xi_2|. 
\end{equation*}
Similarly, the map $\xi\mapsto \frac{\xi}{|\xi|}$ is Lipschitz in $\{|\xi|\ge m_1\}:$  
$$
\begin{aligned}
\Big|\frac{\xi_1}{|\xi_1|} - \frac{\xi_2}{|\xi_2|}\Big| \le &
\frac{|\xi_1 -\xi_2|}{|\xi_1|} + \frac{\big||\xi_1| - |\xi_2|\big|}{|\xi_1|} \le \frac{2|\xi_1 -\xi_2|}{|\xi_1|}
\end{aligned}
$$
so that 
\begin{equation*}
\Big|\frac{\xi_1}{|\xi_1|} - \frac{\xi_2}{|\xi_2|}\Big| \le  \frac{2}{ m_1 }\,|\xi_1 - \xi_2|. 
\end{equation*}

Let 
$$
E:=\{ m_1  \le |\xi|\le m_2\} 
$$
and let us estimate $\|a\|_\infty$ and $[a]_{\alpha}$ in $E.$ Obviously,
$$
\|a\|_\infty \le \frac{\|\psi\|_\infty}{m_1} 
$$
and
\begin{align*}
|a(\xi') - a(\xi'')| \le & \big||\xi'|^{-2} - |\xi''|^{-2}\big|\, \psi\big(\tfrac{\xi'}{|\xi'|}\big) + |\xi''|^{-2} \Big[\psi\big(\tfrac{\xi'}{|\xi'|}\big)-\psi\big(\tfrac{\xi''}{|\xi''|}\big)\Big]\\[2mm]
\le & \frac{2\|\psi\|_\infty}{m_1^3}\,|\xi'-\xi''|
+ \frac{2\|\nabla\psi\|_\infty }{m_1^3}\,|\xi'-\xi''|\\[2mm]
= &\frac{2\|\psi\|_\infty+2\|\nabla\psi\|_\infty }{m_1^3}\,|\xi'-\xi''|
\end{align*}
so that 
$$
\|a(v') - a(v'')\|_\infty \le \frac{2\|\psi\|_\infty+2\|\nabla\psi\|_\infty }{m_1}\,\|v'-v''\|_\infty.
$$
Moreover, since $\psi$ is $C^{2+\alpha}$ in a tubular neighborhood of $\S^1,$ the function
$$
\nabla  a(\xi)= |\xi|^{-5}\big[\nabla \psi(\xi/|\xi|)\cdot \xi^c\big] \xi^c - 4|\xi|^{-4}\psi(\xi/|\xi|) \xi,
$$
is Lipschitz, where $\xi^c=(\xi_2,\xi_1)\in{\mathbb R}^2.$ Its Lipschitz constant does not exceed 
$$
L_1:=\frac{5\|\nabla\psi\|_\infty + 2\|\nabla^2\psi\|_\infty}{ m_1 ^6}\,m_2^2 + \frac{16\|\psi\|_\infty + 10\|\nabla\psi\|_\infty}{ m_1 ^5}\,m_2 + \frac{4\|\psi\|_\infty}{ m_1 ^4}.
$$
Now by Lemma \ref{lem:Holder_norm_of_composition} applied with $a,$ $\rho=1$ and $\rho'=\alpha$ and $\rho'=\alpha/2$ we get
\begin{align*}
[a(v') - a(v'')]_{\alpha,x} \le \|a\|_\infty [v'-v'']_{\alpha,x} 
+ \|\nabla a\|_\infty \max\{[v']_{\alpha,x},[v'']_{\alpha,x}\}\|v'-v''\|_\infty
\end{align*}
and 
\begin{align*}
[a(v') - a(v'')]_{\alpha/2,t} \le \|a\|_\infty [v'-v'']_{\alpha/2,t} \\
+ \|\nabla a\|_\infty \max\{[v']_{\alpha/2,t},[v'']_{\alpha/2,t}\}\|v'-v''\|_\infty,
\end{align*}
and hence, estimates \eqref{a1_holder_zest01}-\eqref{a1_holder_test01} follow with 
\begin{equation*}
L: = \max\big\{\|\psi\|_\infty/m_1, \, L_1\big\}. 
\end{equation*}

The proof of \eqref{ai_holder_est009} is similar.
\end{proof}

\end{document}